\documentclass[11pt]{amsart}
\usepackage{amssymb}
\usepackage{amscd}
\usepackage[all]{xy}
\newtheorem{thm}{Theorem}[section]
\newtheorem{cor}[thm]{Corollary}

\newtheorem{lemma}[thm]{Lemma}
\newtheorem{prop}[thm]{Proposition}

\theoremstyle{definition}
\newtheorem{defn}[thm]{Definition}

\theoremstyle{remark}

\newtheorem{remark}[thm]{Remark}

\begin{document}

\title{Motivic Weight Complexes for Arithmetic Varieties}
\author{Henri Gillet }
\address{Department of Mathematics, Statistics, and Computer Science\\
University of Illinois at Chicago\\
322 Science and Engineering Offices (M/C 249)\\
851 S. Morgan Street\\
Chicago, IL 60607-7045\\
U.S.A.}
\email{gillet@uic.edu}

\author{Christophe Soul\'e}
\address{Institut des Hautes \'Etudes Scientifiques\\
35 route de Chartres\\
91440 Bures-sur-Yvette\\
France}
\email{soule@ihes.fr}

\dedicatory{To Paul Roberts}
\date{\today}
\thanks{This research partially supported by NSF grant DMS-0500762}
\date{\today}


\maketitle

\tableofcontents

\newpage

\section*{Introduction}

In this paper we extend  the results in the paper \cite{Gillet-Soule-motives-descent} to arithmetic varieties.  While in {\it op. cit.}, our main results were framed in terms of Chow motives with integral coefficients, in this paper we shall use $K_0$-motives with rational coefficients.  This is because we do not have resolution of singularities, but rather must appeal to De Jong's results in \cite{de-jong-fourier}.

In addition, De Jong's results in \cite{de-jong-fourier} lead us to extend the theory of weight complexes to Deligne-Mumford stacks. However this is not a substantial generalization, since the weight complexes of a stack and of its ``coarse space'' are homotopy equivalent.

In this paper we shall prove the following theorem, where $S$ is a base scheme satisfying the condition (C) below.

\begin{thm}
\label{motstacks}
There is a covariant functor $h : \mbox{\rm Stack}_S \to \mbox{\rm Ho} \, (C_* (K(S))$ from the category of (separated)
Deligne-Mumford stacks of finite type over $S$,
to the category of homotopy classes of maps of bounded complexes of (homological) $K_0$-motives over $S$ with
rational coefficients, having the following properties:
\begin{itemize}
\item If $X$ is a regular scheme, projective over $S$, then $h(X)$ is the usual motive of $X$.
\item If $X$ is a regular scheme, projective over $S$, and $G$ is a finite group acting on $X$, then $h([X/G]) = h(X)^G$.
Here $[X/G]$ is the quotient stack associated to the action.
\item If $\mathfrak{Y} \subset \mathfrak{X}$ is a closed substack with complement $\mathfrak{U}$, then we have a triangle
$$
h(\mathfrak{Y}) \to h(\mathfrak{X}) \to h(\mathfrak{U}) \to h(\mathfrak{Y}) \, [+ 1] \, .
$$
\end{itemize}
\end{thm}
An immediate consequence of the theorem is:
\begin{cor}\label{Euler}
One can associate to any reduced separated Deligne-Mumford stack $\mathfrak{X}$
of finite type over $S$ an
element $\chi_c(\mathfrak{X})$ in the Grothendieck group $K_0(\mathbf{KM}_S)$ of the category $\mathbf{KM}_S$ of $K_0$-motives over $S$, with the following properties:
\begin{enumerate}
 \item[(i)] If $X$ is a regular projective scheme over $S$ equipped with an action by a finite group $G$ then
$\chi_c([X/G]) $ is the class of
$$(X,\frac{1}{\#(G)}\sum_{g\in G}[\mathcal{O}_{\Gamma(g_*)}]);$$ 
here $[X/G]$ is the quotient stack associated to the action of $G$ on $X$ and $\Gamma(g_*)$ is the graph of the action $g_*:X\to X$ of an element $g\in G$.
\item[(ii)] If $\mathfrak{Y}\subset\mathfrak{X}$ is a closed substack, with complement
$\mathfrak{X}\setminus\mathfrak{Y}$,
$$\chi_c(\mathfrak{X})=\chi_c(\mathfrak{Y})+\chi_c(\mathfrak{X}\setminus\mathfrak{Y}) .$$
\end{enumerate}

\end{cor}

We work over a fixed base scheme $S$ which is regular, excellent, and finite dimensional.
Such an $S$ is sufficient to define the category of $K_0$-motives. However, in order to apply the results of
de Jong to construct weight complexes, we must additionally assume:
\begin{itemize}
\item[(C)] for every finite morphism $\pi : T \to S$, and finite group $G$ acting on $T$ over $S$, the pair $(T,G)$
satisfies 5.12.1 of \cite{de-jong-fourier}.
\end{itemize}

\noindent It is straightforward to check that examples of such an $S$ are:
\begin{itemize}
\item[i)] $S = {\rm Spec} (k)$ with $k$ a field
\item[ii)] $S = {\rm Spec} (\Lambda)$ with $\Lambda$ an excellent Dedekind domain. In particular $S = {\rm Spec} ({\mathcal O}_K)$ with ${\mathcal O}_K$ the ring of integers in a number field.
\end{itemize}

This article is organized as follows.

In the first section, after deriving some basic properties of algebraic stacks,
we deduce from \cite{de-jong-fourier} two results about resolution of singularities,
one for arbitrary Deligne-Mumford stacks over a base scheme $S$ satisfying the condition (C) above, and one about quotient stacks.

In the second paragraph, we discuss simplicial schemes and hypercovers.
We introduce the notion of proper hypercover of a stack ${\mathfrak X}$ by a simplicial scheme and derive some of its properties,
{\it e.g.} the fact that is stable by base change by any morphism from a simplicial scheme to
${\mathfrak X}$ (Lemma~\ref{base change hypercovers stacks}).
We show that any simplicial variety admits a split proper hypercover (Prop.~\ref{split})
and we also show that these coverings can be compactified (Prop.~\ref{compact}).

The next section discusses homological descent, following the method of
SGA4 \cite{SGA4-2}, for arbitrary covariant functors from the category
of proper morphisms between schemes to the category of connective spectra
which satisfy appropriate axioms. This result applies in particular to rational
$G$-theory. The main statement, Theorem~\ref{descent}, says that any
proper hypercover between simplicial varieties with proper face maps
induces a weak equivalence of the associated spectra with rational coefficients. We also show, in 
Theorem~\ref{descent_hyperenvelope}, that any
hyperenvelope between simplicial varieties with proper face maps
induces a weak equivalence of the associated spectra with integral coefficients.

The fourth section is a review of the $G$-theory of stacks and
simplicial schemes. In particular we show that $G$-theory may be
extended to a covariant functor on the category of \emph{all}
proper, not necessarily representable, morphisms between
Deligne-Mumford stacks. The next section contains the main results
about weight complexes. After defining correspondences and
homological motives over $S$, we associate a weight complex of
motives $\Gamma_* (\alpha.)$ to an arbitrary arrow $\alpha.$ of
simplicial projective varieties (Theorem~\ref{maps}). We use this
result to define weight complexes for arbitrary simplicial varieties
over $S$ (Theorem~\ref{simplicial})and then to prove Theorem \ref{motstacks} and Corollary
\ref{Euler}.  We then show how this implies the existence
of a virtual Chow motive (with rational coefficients) for every
variety over a base field of arbitrary characteristic.
Finally using a result of O.~Gabber, we
show that given a prime $\ell$ invertible in $S$ one can define
weight complexes for Deligne-Mumford stacks over $S$ with values in the category of homotopy
classes of complexes of $K_0$-motives over $S$ with ${\mathbb Z}_{(\ell)}$-coefficients.
However, we do not know if this latter construction leads to bounded complexes or not.

The last section is devoted to the proof of contravariance properties of weight complexes. We assume that $S$ is the spectrum of a field.
Given two varieties $X$ and $Y$ over $S$ we define a Waldhausen category ${\mathbf C}(X,Y)$
of complexes of sheaves on $X \times Y$ and we prove that it is contravariant in $X$ and covariant in $Y$ (Lemma~\ref{contra}).
When $X$ is regular  and projective the $K$-theory $\mathbf{KC} (X,Y)$ of ${\mathbf C} (X,Y)$ coincides
with the $G$-theory of $X \times Y$ (Proposition~\ref{G}).
We then extend the construction to an arbitrary pair of maps $(\alpha. , \beta.)$ of
simplicial varieties, getting a spectrum $\mathbf{KC} (\alpha. , \beta.)$.
When the simplicial varieties involved are regular and projective,
we define a morphism $\gamma$ from $\pi_0 \, \mathbf{KC} (\alpha. , \beta.)$
to the group of homotopy classes of maps from $\Gamma_* (\alpha.)$ to $\Gamma_* (\beta.)$.
Using this, given two projective varieties $X$ and $Y$ we define a map from $\pi_0 \, \mathbf{KC} (X,Y)$ to
${\rm Hom} (h(X) , h(Y))$ which extends earlier constructions, and in particular,
we attach to any morphism $f : X \to Y$ of finite tor-dimension between varieties,
a morphism of weight complexes $f^* : h(Y) \to h(X)$.
When $f$ is an open immersion, we show that $f^*$ coincide with the map already defined.

We would like to thank the referee for useful comments on the paper.

\noindent {\bf Conventions.}
\begin{itemize}
\item All schemes and stacks will be assumed to be separated.
\item By a variety we will simply mean a scheme which is of finite type over our base $S$. The category of such varieties and proper morphisms between them will be denoted ${\mathcal V}ar_S$.
\item A scheme is said to be integral if it is reduced and irreducible.
\item A simplicial scheme $X.$ over $S$ will be said to be {\it proper} over $S$, if each $X_n$ is proper over $S$.
\item A subsimplicial scheme $U. \subset X.$ will be said to be {\it strongly open} if its complement $X.\setminus U.$ is a closed sub-simplicial scheme. This is equivalent to requiring that $U. = {\rm cosk}^{X.} (U_0)$.
\end{itemize}

\newpage

\section{Stacks}

\subsection{Quotient stacks}

For definitions and terminology relating to stacks see \cite{Laumon-M-B-stacks}.

All our stacks will be Deligne-Mumford stacks of finite type over $S$.

If $\mathfrak{X}$ is a stack, we denote its set of points by $\vert \mathfrak{X} \vert$, which is a finite dimensional noetherian topological space (\cite{Laumon-M-B-stacks}, Chapter 5).  Given $x \in \vert \mathfrak{X} \vert$, it has a residue field ${\mathbf k} (x)$ which is an $S$-field ({\it op. cit.} 11.2). We write $\mathfrak{X}_{(i)}$ for the set of points $x \in \vert \mathfrak{X} \vert$ which are the generic points of $i$-dimensional subsets.

Following \cite{Keele-Mori}, \cite{Conrad} we know that $\vert \mathfrak{X} \vert$ is the set of points of an algebraic space, which we shall also denote $\vert \mathfrak{X} \vert$ and call the coarse space of ${\mathfrak X}$.

\begin{prop} \label{etale-implies-galois}
{\rm (\cite{Laumon-M-B-stacks} thm. 6.1.)} If $\mathfrak{X}$ has a finite \'etale cover by a variety, then there exists a Galois cover by a variety $\pi : V \to \mathfrak{X}$. \emph{I.e.}, there is a finite group $G$ acting on $V$, and $\mathfrak{X} = [V/G]$ is the quotient stack for the action.
\end{prop}

\begin{defn}
We shall refer to stacks of the form $[V/G]$ as \emph{quotient stacks}.
\end{defn}

Putting together \ref{etale-implies-galois} and \cite{Laumon-M-B-stacks}, corollaire 6.6.1, we get:

\begin{prop}
\label{Proposition113}
Let $\mathfrak{X}$  be a Deligne-Mumford stack.  Then there is a non empty open substack $\mathfrak{U}\subset \mathfrak{X}$ which is a quotient stack, and we can choose $\mathfrak{U}$ so that
$\mathfrak{U} \simeq [U/G]$ with $U$ quasi-projective over $S$.
\end{prop}

\begin{lemma} \label{compactify quotient stacks}
Let $\mathfrak{Z} = [Z/G]$ for  $G$ a finite group acting on an $S$-variety $V$. Then there is a $G$-equivariant compactification $Z \subset W$ of $Z$ over $S$, and hence an open immersion
$\mathfrak{Z} \to \mathfrak{Y} = [W/G]$ with $\mathfrak{Y}$ proper over $S$.
\end{lemma}

\begin{proof} It will be sufficient to find a $G$-equivariant compactification of $Z$.  Recall the following standard argument. Let $i : Z \to \bar{Z}$ be any compactification of $Z$. (This exists by Nagata's theorem.) Consider the $G$-fold fiber product $\bar{Z}^G_S$, which is proper over $S$.  There is a natural morphism of schemes $\eta : Z \to \bar{Z}^G_S$, the ``$g$-th'' component, for $g \in G$, of which is equal to $i \cdot \rho_g : Z \to \bar{Z}$, where $\rho_g : Z \to Z$ is the action of $g$. Observe that  $G$ acts on $\bar{Z}^G_S$ by permuting the factors in the product, and, by construction, $\eta$ is $G$-equivariant. Hence the Zariski closure $W$ of $Z$ in $\bar{Z}^G_S$ is a $G$-equivariant compactification of $Z$. \end{proof}

\begin{lemma}
A map $f : X \to Y$ of quotient stacks is proper if and only if there are finite \'etale covers $U \to X$ and $V \to Y$ by varieties, and a map $\tilde{f} : U \to V$ which is proper, such that the diagram
$$
\begin{CD}
U @>>> V \\
@VVV @VVV\\
X @>>> Y
\end{CD}
$$
commutes.
\end{lemma}

\begin{proof} Let $V \to Y$ be any finite \'etale cover of $Y$ by a variety. Since $f$ being proper is local in the \'etale topology of $Y$, $f$ is proper if and only if $f_V : V \times_Y X \to V$ is proper. Now let
$W \to X$ be a finite \'etale cover of $X$ by a variety. Then the induced map $g : U := V \times_Y W \to V\times_Y X$ is finite, \'etale, and surjective, and so $f_V$ is proper if and only if the composition $\tilde{f}= f_V \circ g : U \to V$ is proper. \hfil $\Box$
\end{proof}

\begin{defn}
Recall that a map $f : X \to Y$ of schemes is \emph{radicial} if equivalently:
\begin{itemize}
\item For every field $F$, the induced map $X(F) \to Y(F)$ is injective.
\item $f$ is injective as a map of schemes, and for every point $x \in X$, the field extension ${\mathbf k} (f(x))\subset {\mathbf k} (x)$ is purely inseparable.
\end{itemize}

Following {\rm \cite{Laumon-M-B-stacks} 3.10}, we say that a representable morphism $f : \mathfrak{X} \to \mathfrak{Y}$ between Deligne-Mumford stacks is radicial if there is an \'{e}tale cover $p : U \to
\mathfrak{Y} $ with $U$ a scheme, such that $U \times_{\mathfrak{Y}} \mathfrak{X}$ is a scheme, and
$U \times_{\mathfrak{Y}} {\mathfrak X}\to U$ is radicial.
\end{defn}

\subsection{Resolution of singularities}

\begin{thm}
\label{proper} If ${\mathfrak X}$ is a Deligne-Mumford stack of finite type over $S$, there is a proper surjective morphism $p : X \to {\mathfrak X}$ with $X$ a regular variety over $S$.
\end{thm}

\begin{proof} By Chow's lemma (\cite{Laumon-M-B-stacks}
Corollaire~16.6.1) there is a proper surjective morphism
$p_0 : X_0 \to {\mathfrak X}$ with $X_0$ a projective variety over $S$. Hence
it suffices  to show that there is a proper surjective map
$\pi : X \to X_0$ such that $X$ is regular. Let $f : X_0 \to S$ be the
structural map, and $f = q \cdot g$ its Stein factorization, with
$g : X_0 \to T$ and $q : T \to S$. Then $q$ is finite, and so by the
hypothesis (C) on $S$, $T$ satisfies 5.12.1 of
\cite{de-jong-fourier}. Furthermore, by Remark~4.3.4 of EGA III
\cite{EGAIII}, $g$ has geometrically connected fibres. Let $T'$ be
the disjoint union of the irreducible components of $T$, and $X'_0$
the pull back of $X_0$ over $T'$. Then $g' : X'_0 \to T'$ still has
geometrically connected fibres, and so the inverse image by $g'$ of
each component of $T'$ is an irreducible component of $X'_0$. By the
standing assumption on $S$, each component of $T'$ satisfies
\cite{de-jong-fourier} 5.12.1, and hence by \cite{de-jong-fourier}
Theorem~5.13 each component of $X'_0$ admits a non-singular
alteration, and we set $X$ equal to the disjoint union of these.
\end{proof}

We also need resolution of singularities for quotient stacks:

\begin{thm}
\label{quotient} Let ${\mathfrak X} = [X/G]$ be a reduced quotient stack of finite type over $S$. Then there is a family of proper morphisms of quotient stacks
$$
p_i : [Y_i / H_i] \to {\mathfrak X} \, ,
$$
where each $Y_i$ is regular and integral, and a dense open substack ${\mathfrak U} \subset {\mathfrak X}$ such that, if
$$
p : \underset{i}{\amalg} \ [Y_i / H_i] \to {\mathfrak X}
$$
is the disjoint union of the $p_i$'s, the induced map
$$
p^{-1} ({\mathfrak U}) \to {\mathfrak U}
$$
is representable and radicial.
\end{thm}

\begin{proof} If ${\mathfrak X} = \underset{i}{\cup} \, {\mathfrak X}_i$ is the decomposition of ${\mathfrak X}$ into irreducible components, since the map $\underset{i}{\amalg} \, {\mathfrak X}_i \to {\mathfrak X}$ is an isomorphism on a dense open subset of both source and target,
it suffices to consider the case when ${\mathfrak X}$ is irreducible.
Since ${\mathfrak X}$ is irreducible, $G$ acts transitively on the irreducible components of $X$.

Let $X_0$ be one such component, and $G_0 \subset G$ its stabilizer. Then since the inclusion $i : X_0 \to X$ is $G_0$-equivariant, there is an induced map
$$
f : [X_0 / G_0] \to [X/G]
$$
which is clearly surjective (since the pull back of the \'etale cover $X \to [X/G]$ by $f$ is isomorphic to the disjoint union of the irreducible components of $X$). Furthermore, since the map from the disjoint union of the irreducible components of $X$ to $X$ induces an isomorphism between dense open subsets of source and target, the same is true for $f$.

It suffices, therefore, to show that there is a map ${\mathfrak Y} = [Y/H] \to [X_0 / G_0]$ with the desired properties.

Consider the map $g_0 : X_0 \to S$. This is $G_0$-equivariant, with $G_0$ acting trivially on $S$. Let
$$
X_0 \overset{q}{\longrightarrow} T \overset{\pi}{\longrightarrow} S
$$
 be its Stein factorization, which is again $G_0$-equivariant. Then the morphism $q$ has geometrically irreducible fibres, and $(T,G_0)$, by the standing assumption on $S$, satisfies 5.12.1 of \cite{de-jong-fourier}, hence ({\it op. cit.}, 5.13) there is a Galois alteration $(Y,H) \to (X_0 , G_0)$ with $Y$ regular. Set ${\mathfrak Y} = [Y/H]$. The field extension $k({\mathfrak Y}) / k({\mathfrak X})$ is purely inseparable by ({\it op. cit.}, 5.13). Since $Y$ is integral, we may assume that the kernel of $H \to G_0$ is the Galois group of this extension. Therefore there is a dense open substack ${\mathfrak U} \subset {\mathfrak X}$ such that the map $p^{-1} ({\mathfrak U}) \to {\mathfrak U}$ is representable and radicial. This proves our result. \end{proof}

\section{Hypercovers}

The material in this section is based on expos\'e $\mathrm{V}^{\mathrm{bis}}$  of \cite{SGA4-2}, as well as \cite{Deligne-Hodge-III}.  Throughout this section $\mathcal{C}$ will be a category with finite limits and colimits. In particular ${\mathcal C}$ has a final object.

\subsection{Simplicial objects and coskeleta}

Recall that $\Delta$ is the category of finite non-empty totally ordered sets and order preserving maps.  For $n\in \mathbb{N}$, we write $[n] := \{0< \ldots <n\}$; these objects (and the morphisms between them) form a skeletal subcategory of $\Delta$.
Let $\Delta_{\leq n} \subset  \Delta$ be the full subcategory consisting of objects with cardinality at most $n$.

A simplicial object in $\mathcal{C}$ is a {\it contravariant} functor $X. : \Delta \to \mathcal{C}$.  We write $X_n$ for $X. ([n])$.
The category of simplicial objects in $\mathcal{C}$ will be denoted by ${\Delta^{\mathrm{op}}} (\mathcal{C})$.

An $n$-truncated simplicial object is a contravariant functor $X. : \Delta_{\leq n} \to \mathcal{C}$.  The corresponding category will be denoted ${\Delta_{\leq n}^{\mathrm{op}}} (\mathcal{C})$.

There is an obvious restriction functor
$$
\mathrm{sk}_n : {\Delta^{\mathrm{op}}} (\mathcal{C}) \to {\Delta_{\leq n}^{\mathrm{op}}} (\mathcal{C}) \; .
$$
This functor has two adjoints.  Since $\mathcal{C}$ has finite inverse limits, ${\rm sk}_n$ has a right adjoint
$$
{\rm cosk}_n : {\Delta_{\leq n}^{\mathrm{op}}} (\mathcal{C}) \to {\Delta^{\mathrm{op}}} (\mathcal{C}) \; .
$$
Note that ${\rm sk}_n \circ {\rm cosk}_n = {\rm Id}$.
In addition, ${\rm sk}_n$ has a left adjoint which we denote $\iota_n$, and which is a fully faithful functor.

We shall write ${\rm Cosk}_n$ for the composition ${\rm cosk}_n \circ {\rm sk}_n$, and ${\rm Sk}_n$ for the composition $\iota_n \circ {\rm sk}_n$.
Note that $({\rm Sk}_n , {\rm Cosk}_n)$ are an adjoint pair.

If $Y.$ is a fixed object in ${\Delta^{\mathrm{op}}} (\mathcal{C})$, we can consider the category
${\Delta^{\mathrm{op}}} (\mathcal{C})_{Y.}$ of simplicial objects over $Y.$, as well as the category of $n$-truncated objects ${\Delta_{\leq n}^{\mathrm{op}}} (\mathcal{C})_{Y.}$ over ${\rm sk}_n(Y.)$.

There is an obvious restriction functor
$$
{\rm sk}^{Y.}_n : {\Delta^{\mathrm{op}}} (\mathcal{C})_{Y.} \to {\Delta_{\leq n}^{\mathrm{op}}} (\mathcal{C})_{Y.}
$$
which has a right adjoint ${\rm cosk}^{Y.}_n$, given by
$$
{\rm cosk}_n^{Y.} (X.) = {\rm cosk}_n (X.) \times_{{\rm Cosk}_n {(Y.)}} Y.
$$
We write ${\rm Cosk}^{Y.}_n$ for the composition ${\rm cosk}^{Y.}_n \circ {\rm sk}^{Y.}_n$.  We also set ${\rm Cosk}^{Y.}_{-1} := Y.$

The following is straightforward, so we omit the proof.

\begin{lemma}\label{props-cosk}
Let $f : X. \to Y.$ be a map of simplicial objects in ${\mathcal C}$.
\begin{itemize}
\item[{\rm 1)}] If $m \geq n \geq 0$, then the natural map
$$
{\rm Cosk}_n^{Y.} (X.). \to {\rm Cosk}_m^{Y.} ({\rm Cosk}_{n.}^{Y.} (X.).).
$$
is an isomorphism.
\item[{\rm 2)}] If $n \geq p \geq 0$, the natural map
$$
X_p \to {\rm Cosk}_n^{Y.} (X.)_p
$$
is an isomorphism.
\item[{\rm 3)}] If $n > p \geq 0$, then
$$
{\rm Cosk}_n^{Y.} (X.)_p \to {\rm Cos}_{n-1}^{Y.} (X.)_p
$$
is an isomorphism.
\end{itemize}
\end{lemma}

If $A$ is a finite set, and $X \in \mathcal{C}$ is an object,  we write $X \times A$ for $\sqcup_{a \in A} \, X$. If $A.$ is the simplicial set associated to a finite simplicial complex, and $X. \in \Delta^{\mathrm{op}} \mathcal{C}$ we define $X. \times A.$ by $(X. \times A.)_k := X_k \times A_k$.

Two maps $f_0 : X. \to Y.$ and $f_1 : X. \to Y.$ are said to be homotopic if there is a map
$$
h. : X. \times \Delta [1] \to Y.
$$
such that $h. |_{X \times \{i\}} = f_i$.

\begin{lemma}
\label{Lemme212}
Let $X$ and $Y$ be two objects of $\mathcal{C}$.  If $f_i : X \to Y$ for $i = 0,1$ are two arbitrary maps, then the maps ${\rm cosk} (f_0) : {\rm cosk}_0(X). \to {\rm cosk}_0(Y).$  and ${\rm cosk} (f_1) : {\rm cosk}_0(X). \to {\rm cosk}_0 (Y).$ are homotopic.
\end{lemma}

\begin{proof} We have to construct maps, for $p\geq 0$:
$$
({\rm cosk}_0(X)_p = X^{p+1})\times  \, \mathrm{Hom}_{\Delta} ([p],[1])  \to {\rm cosk}_0 (Y)_p = Y^{p+1}
$$
which are compatible with the morphisms in $\Delta$. Given $\phi : [p] \to [1]$, define
$$
h_p (\phi) := (h_p(\phi)_0 , \ldots , h_p (\phi)_p) : X^{p+1} \to Y^{p+1}
$$
by, for $i = 0, \ldots , p$,
$$
h_p (\phi)_i = f_{\phi(i)} : X \to Y \, .
$$
It is straightforward to check that this works. \end{proof}

More generally, we have the following result, which extends Lemme~3.0.2.4 of expos\'e V$^\textrm{bis}$ of \cite{SGA4-2}:

\begin{lemma}
\label{existence-homotopies}
Let $S.$ be a fixed simplicial object in $\mathcal{C}$.
Suppose that $n \geq 0$ and that $f_i : X. \to Y.$ for $i = 0,1$ are
two maps in the category ${\Delta^{\mathrm{op}}_{S.}} (\mathcal{C})$ of
simplicial objects over $S.$, such that $(f_0)_p = (f_1)_p$ for $p < n$.
Note that if $n = 0$ this last condition is vacuous.
Then ${\rm Cosk}^{S.}_n (f_0)$ and ${\rm Cosk}^{S.}_n (f_1)$ are homotopic.
\end{lemma}

\begin{proof} The two maps ${\rm Cosk}^{S.}_n (f_0).$ and ${\rm Cosk}^{S.}_n (f_1).$ are
homotopic if there is a map
$$
h. : {\rm Cosk}^{S.}_n (X). \times \Delta [1] \to {\rm Cosk}^{S.}_n (Y).
$$
such that $h. |_{{\rm Cosk}^{S.}_n (X) \times \{i\}} = f_i$. Since
$$
{\rm Cosk}^{S.}_n (X). = {\rm Cosk}_n (X). \times_{{\rm Cosk}_n(S).} S.
$$
it is enough to find a homotopy
$$
h. : {\rm Cosk}_n (X). \times \Delta [1] \to {\rm Cosk}_n (Y).
$$
which is a map of simplicial objects over ${\rm Cosk}_n (S).$

To give the map $h.$ is equivalent to giving a map
$$
{\rm sk}_n (h.) : {\rm sk}_n (X.) \times {\rm sk}_n (\Delta [1]) \to {\rm sk}_n (Y.) \, ,
$$
{\it i.e.} to giving maps, for $k \leq n$,
$$
h_{k,\phi} : X_k \to Y_k
$$
for $\phi \in \Delta [1]_k = \mathrm{Hom}_\Delta ([k],[1])$ compatible with faces and degeneracies, such that $h_{k,\partial_i} = (f_i)_k$ for $\partial_i : [k] \to [1] = \{0,1\}$ the constant map with value $i$.

Since $(f_0)_k = (f_1)_k$ for $k < n$, we can set $h_{k,\phi} = (f_0)_k = (f_1)_k$ for $k < n$.  For $k = n$, and $\phi \neq \partial_0, \partial_1$, we can choose $h_{k,\phi} = f_0$ or $f_1$ arbitrarily.  It is now straightforward to check that such a choice defines a homotopy, and that since $f_0$ and $f_1$ are maps of simplicial objects over $S.$, $h.$ is a map of objects over ${\rm Cosk}_n(S)$. \end{proof}

If ${\mathcal C}$ is a category with finite products, if $A$ is a finite set, and $X$ is an object of ${\mathcal C}$ we define ${\rm Hom} (A,X) := X^A$. Notice that this is a functor
$$
\mbox{Finite Sets}^{\rm op} \times {\mathcal C} \to {\mathcal C} \, .
$$
Let $n \geq 0$. Given an $n$-truncated simplicial finite set $A.$ and $n$-truncated simplicial object
$X.$ in ${\mathcal C}$, ${\rm Hom} (A. , X.)$ defines a
functor $(\Delta_{\leq n})^{\rm op} \times (\Delta_{\leq n}^{\rm op})^{\rm op} \to {\mathcal C}$ and we define
$$
{\rm Hom}_{\Delta_{\leq n}} (A. , X.) := \int_{[p] \in \Delta_{\leq n}} {\rm Hom} (A_p , X_p) \in {\mathcal C} \, .
$$
Here $\int_{[p] \in \Delta_{\leq n}}$ denotes the end as in \cite{MacLane-Categories-working-math}.
We say that a simplicial set $A.$ is {\it finite} if all the $A_i$ are finite, and there is an
$n \geq 0$ such that $A. = {\rm Sk}_n (A.)$. Notice that if $K$ is a finite simplicial complex, the associated simplicial set is finite.

If $A.$ is a finite simplicial set with $A. = {\rm Sk}_n \, A.$, and $X.$ is a simplicial object in ${\mathcal C}$, we define:
$$
{\rm Hom}_{\Delta} (A. , X.) := {\rm Hom}_{\Delta_{\leq n}} ({\rm Sk}_n \, A. , {\rm Sk}_n \, X.) \, .
$$
Note that if ${\mathcal C}$ had infinite products so that we could define ${\rm Hom}_{\Delta} (A. , X.)$ directly, then we would have, since $A. = {\rm Sk}_n \, A. = \iota_n \circ {\rm sk}_n \, A$:
\begin{align}
{\rm Hom}_{\Delta} (A. , X.) &= &{\rm Hom}_{\Delta} (\iota_n \cdot {\rm sk}_n \, A , X.) \nonumber \\
&= &{\rm Hom}_{\Delta_{\leq n}} ({\rm sk}_n \, A , {\rm sk}_n \, X.) \, . \nonumber
\end{align}
It is straightforward to check that this definition does not depend on the choice of $n$, so long as $A. = {\rm Sk}_n \, A$.

\begin{prop}
\label{left_exact}
The functor ${\rm Hom}_{\Delta} : [(\Delta^{\rm op} \, {\rm Sets})_{\rm finite}]^{\rm op} \times \Delta^{\rm op} \, {\mathcal C} \to {\mathcal C}$ is left exact with respect to the first variable, \emph{i.e.}, suppose that $A : I \to (\Delta^{\mathrm{op}} \, {\rm Sets})_{\rm finite}$ is a diagram with $I$ finite, then
$$
\mathrm{Hom}_\Delta (\varinjlim_{i\in I} A(i)., X.) \simeq \varprojlim_{i\in I} \mathrm{Hom}_\Delta(A(i)., X.) \, .
$$
\end{prop}

\begin{proof} First observe that
$$
A \mapsto X^A
$$
is left exact with respect to $A$, and that direct limits in $(\Delta^{\mathrm{op}} \, {\rm Sets})_{\rm finite}$ are computed degreewise.  Since $I$ is finite, there is an $n \geq 0$ such that
$$
\varinjlim_{i \in I} \, {\rm Sk}_n (A(i).) = {\rm Sk}_n (\varinjlim_{i \in I} A(i).) \, .
$$
Therefore
\begin{eqnarray*}
{\rm Hom}_{\Delta} (\varinjlim_{i \in I} A(i). , X.) &:= &\int_{p \in \Delta_{\leq n}} {\rm Hom} \, (\varinjlim_{i \in I} A(i)_p , X_p) \nonumber \\
&= &\int_{p \in \Delta_{\leq n}} {\rm Hom}_{\Delta_{\leq n}} ({\rm sk}_n (\varinjlim_{i \in I} A(i)_p) , {\rm sk}_n (X_p)) \nonumber \\
&= &\int_{p \in \Delta_{\leq n}} \varprojlim_{i \in I} ({\rm Hom} (A(i)_p , X_p)) \nonumber \\
&\simeq &\varprojlim_{i \in I} \int_{p \in \Delta_{\leq n}} {\rm Hom} \, (A(i)_p , X_p) \nonumber \\
&&\mbox{(by the Fubini theorem, \cite{MacLane-Categories-working-math}, IX.8)} \nonumber \\
&= &\varprojlim_{i \in I} {\rm Hom}_{\Delta} (A(i). , X.) \, . \nonumber
\end{eqnarray*}
\end{proof}

\subsection{Hypercovers and right lifting}

Recall that a map $f : X \to Y$ between simplicial sets is said to be a trivial fibration (II.2.2 of \cite{Quillen-Homotopical-Algebra}), if it has the right lifting property for all monomorphisms $i : A. \to B.$ of simplicial sets. That is, for all commutative squares:
$$
\xymatrix{
A.\ar[d]_i\ar[r]^\alpha & X.\ar[d]^f  \\
B.\ar[r]_\beta   & Y. \\
}
$$
there is a lifting $\gamma : B. \to X.\textbf{}$ such that the diagram
$$
\xymatrix{
A.\ar[d]_i\ar[r]^\alpha & X.\ar[d]^f  \\
B.\ar[r]_\beta\ar@{-->}[ur]^\gamma   & Y. \\
}
$$
commutes.  We can remove the explicit mention of the maps $\alpha$ and $\beta$, by defining
$\mathrm{Hom} (A.,X.) \times_{\mathrm{Hom} (A.,Y.)} \mathrm{Hom} (B.,Y.)$ to be the pull back in the square
$$
\xymatrix{
& \mathrm{Hom}(B.,A.)\ar[d]^{i^*}  \\
\mathrm{Hom}(A.,X.)\ar[r]_{f_*}   & \mathrm{Hom}(A.,Y.) \\
}
$$
Then the lifting property becomes the assertion that the natural map
$$
\mathrm{Hom} (B.,X.) \overset{(i^*,f_*)}{\longrightarrow} \mathrm{Hom}(A.,X.)\times_{\mathrm{Hom}(A.,Y.)}\mathrm{Hom}(B.,Y.)$$
is surjective.

Suppose now that $f : X. \to Y.$ is a morphism in $\Delta^{\mathrm{op}} \mathcal{C}$.  If  $i : A. \to B.$ is a monomorphism between finite simplicial sets, we define
$$
\mathrm{Hom}_\Delta (A.,X.) \times_{\mathrm{Hom}_\Delta(A.,Y.)} \mathrm{Hom}_\Delta (B.,Y.)
$$
to be the pull back in the square
$$
\xymatrix{
& \mathrm{Hom}_\Delta(B.,A.)\ar[d]^{i^*}  \\
\mathrm{Hom}_\Delta(A.,X.)\ar[r]_{f_*}   & \mathrm{Hom}_\Delta(A.,Y.) \\
}
$$

Let $\mathcal{P}$ be a class of morphisms in $\mathcal{C}$ which contains isomorphisms,
 is closed under composition,
 and is closed under base change.
It is straightforward to check that these two conditions imply that ${\mathcal P}$ is closed under products.

\begin{defn}
We say that a morphism $f : X. \to Y.$ of simplicial objects in ${\mathcal C}$ has the right ${\mathcal P}$-lifting property with respect to an inclusion $i : A. \hookrightarrow B.$ between finite simplicial sets, if the map
$$
{\rm Hom}_{\Delta} (B. , X.) \to {\rm Hom} (A. , X.) \times_{{\rm Hom} (A. , Y.)} {\rm Hom} (B. , Y.)
$$
is in ${\mathcal P}$.
\end{defn}

Recall that the morphism $f : X \to Y$ in $\Delta^{\mathrm{op}} \mathcal{C}$ is called a $\mathcal{P}$-hypercover, if for all $n \geq 0$, the morphism
$$
X_n \to {\rm Cosk}^{Y.}_{n-1} (X.)_n
$$
is in $\mathcal{P}$. By definition of ${\rm Cosk}^{Y.}_{n-1} (X.)_n$, this is equivalent to $f$ having, for each $n \geq 1$,
the right lifting property with respect to the inclusion
$$
{\rm sk}_{n-1} (\Delta_n) \subset \Delta_n \, ,
$$
and for $n = 0$, this means that $X_0 \to Y_0$ is in ${\mathcal P}$, {\it i.e.} the lifting property with respect to
$\emptyset \hookrightarrow \Delta_0$. Applying II.3.8. of \cite{Gabriel-Zisman} and Proposition
\ref{left_exact}, we obtain:

\begin{lemma}
A map $f : X \to Y$ in $\Delta^{\mathrm{op}} \mathcal{C}$ is a $\mathcal{P}$-hypercover if and only if, for all monomorphisms $i : A. \to B.$ of finite simplicial sets, the natural map
$$
\mathrm{Hom}_\Delta (B.,X.) \overset{(i^*,f_*)}{\longrightarrow} \mathrm{Hom}_\Delta (A.,X.) \times_{\mathrm{Hom}_\Delta(A.,Y.)} \mathrm{Hom}_\Delta (B.,Y.)
$$
is in $\mathcal{P}$, {\rm i.e.} $f$ has the right lifting property with respect to {\rm all} injective maps between finite simplicial sets.
\end{lemma}

From which it immediately follows that:

\begin{cor}
If $f : X \to Y$ in $\Delta^{\mathrm{op}} \mathcal{C}$ is a $\mathcal{P}$-hypercover, then for all $n$, $f_n \in \mathcal{P}$.
\end{cor}

\begin{proof} For any given $n$, this is simply the right lifting property with respect to the inclusion $\phi \hookrightarrow \Delta_n$. \end{proof}
\begin{lemma}
\label{props-tower}
If $f : X. \to Y.$ is a $\mathcal{P}$-hypercover, then the tower of maps:
\begin{multline*}
\ldots \to {\rm Cosk}^{Y.}_{n+1} (X.). \overset{f_{n+1}}{\longrightarrow} {\rm Cosk}^{Y.}_n(X.). \to\\ \ldots {\rm Cosk}^{Y.}_0(X.). \overset{f_0}{\longrightarrow} {\rm Cosk}^{Y.}_{-1} (X.). = Y.
\end{multline*}
has the following properties for all $n \geq 0$:
\begin{itemize}
\item[{\rm 1)}] the natural map
$$
X_p \to {\rm Cosk}_n^{Y.} (X.)_p
$$
is an isomorphism for $p \leq n$.
\item[{\rm 2)}] $(f_n)_p$ is an isomorphism for $p < n$.
\item[{\rm 3)}] $(f_n)_n \in {\mathcal P}$.
\end{itemize}
\end{lemma}

\begin{proof} 1) and 2) are already in Lemma~\ref{props-cosk} while 3) follows from 1) and the definition of a hypercover.\end{proof}

For schemes and stacks of finite type over $S$, we shall be interested in hypercovers with respect to two possible choices of ${\mathcal P}$:
\begin{itemize}
\item[i)] ${\mathcal P}$ consists of those maps $f : X \to Y$ which are {\it envelopes}
{\it i.e.} are proper and surjective on $F$-valued points for all fields $F$.  We shall refer to such hypercovers as \emph{hyperenvelopes}.
\item[ii)] ${\mathcal P}$ consists of all maps $f : X \to Y$ which are proper and surjective,
{\it i.e.}, surjective on $F$-valued points for all {\it algebraically closed} fields. We shall refer to such hypercovers as \emph{proper hypercovers}.
\end{itemize}

The proof of the following proposition is a straightforward consequence of the definition of right ${\mathcal P}$-lifting, and the valuative criterion of properness:

\begin{prop}
Let $f. : X. \to Y.$ be a map of simplicial varieties. Then $f.$ is a proper hypercover if, for all injections $i : A \hookrightarrow B$ between finite simplicial sets, the following two conditions hold:
\begin{itemize}
\item[{\rm i)}] For all algebraically closed fields $F$, and each commutative diagram:
$$
\xymatrix{
A. \times {\rm Spec} (F) \ar@{^{(}->}[d]^{i.} \ar[r]^{\ \ \ \ \ \alpha} & X.\ar[d]^{f.}  \\
B. \times {\rm Spec} (F) \ar[r]^{\ \ \ \ \ \beta} \ar@{-->}[ur]^{\rho.} & Y. \\
}
$$
there is a lifting $\rho. : B. \times {\rm Spec} (F) \to X.$
\item[{\rm ii)}] For every valuation ring $\Lambda$ with fraction field $F$, and each commutative diagram:
$$
\xymatrix{
A. \times {\rm Spec} (\Lambda) \cup B. \times {\rm Spec} (F) \ar@{^{(}->}[d] \ar[r]^{\ \ \ \ \ \ \ \ \ \ \ \ \ \ \ \alpha} & X.\ar[d]  \\
B. \times {\rm Spec} (\Lambda) \ar[r]^{\ \ \ \ \ \ \ \beta}\ar@{-->}[ur]^{\rho.}   & Y. \\
}
$$
there is a lifting $\rho. : B. \times {\rm Spec} (\Lambda) \to X.$
\end{itemize}
$f.$ is a {\rm hyperenvelope} if in {\rm (i)}, the field $F$ is {\rm any} field.
\end{prop}

\subsection{Hypercovers of schemes and stacks}

Recall that a stack ${\mathfrak X}$ over $S$ is in particular a category over the category of $S$-varieties:
$$
\alpha : {\mathfrak X} \to {\mathcal V}ar_S
$$
with $\alpha^{-1} (X)$ the groupoid ${\rm Hom} (X,{\mathfrak X})$. If $f : X \to {\mathfrak X}$ and $g : Y \to {\mathfrak X}$ are two morphisms, then $f \times_{\mathfrak X} g$ is the product of $f$ and $g$ in the category ${\mathfrak X}$. Then $\alpha (f \times_{\mathfrak X} g)$ is the scheme parametrizing isomorphisms:
$$
\theta : (f \circ p_X = p_X^* (f)) \simeq (g \circ p_Y = p_Y^* (g))
$$
where $p_X : X \times_S Y \to X$ and $p_Y : X \times_S Y \to Y$ are the projections.

\begin{defn}
A morphism from a simplicial scheme to a stack ${\mathfrak X}$, $f. : X. \to {\mathfrak X}$, consists of a simplicial object $f.$ in ${\mathfrak X}$, such that $\alpha (f.) = X.$
\end{defn}

\begin{lemma}
Given $f.$ as above, let ${\rm cosk}_0^{\mathfrak X} (X_0).$ be the simplicial scheme $n \mapsto X_0 \times_{\mathfrak X} \ldots \times_{\mathfrak X} X_0 = \alpha (f_0 \times_{\mathfrak X} \ldots \times_{\mathfrak X} f_0)$, in which ${\rm cosk}_0^{\mathfrak X} (X_0)_n$ parametrizes $n$-tuples
$$
(\theta_{0,1} , \ldots , \theta_{n-1,n})
$$
with $\theta_{i,i+1} : f_0 \to f_0$ an isomorphism in the groupoid $\alpha^{-1} (X) = {\rm Hom} (X,{\mathfrak X})$. Then giving the full simplicial object $f.$ in ${\mathfrak X}^{\Delta^{\rm op}}$ is equivalent to giving a map of simplicial schemes
$$
X. \to {\rm cosk}_0^{\mathfrak X} (X_0) \, .
$$
\end{lemma}

\begin{proof} A map
$$
X_n \to {\rm cosk}_0^{\mathfrak X} (X_0)_n
$$
is equivalent to giving a map
$$
\varphi_n : X_n \to X_0 \times_S \ldots \times_S X_0 = {\rm cosk}_0^S (X_0)_n
$$
and isomorphisms, for $i = 1 \ldots n$
$$
\theta_{i-1,i} : (p_{i-1} \cdot \varphi_n)^* \, (f_0) \to (p_i \cdot \varphi_n)^* (f_0) \, .
$$
This is equivalent, up to canonical isomorphism, to giving an object $f_n \in {\mathfrak X}$ with $\alpha (f_n) = X_n$ and isomorphisms, for $i = 0, \ldots , n$:
$$
\theta_i : f_n \to (p_i \, \varphi_n)^* \, (f_0) \, .
$$
Given $f_n$, and the $\theta_i$, set $\theta_{i-1 , i} = \theta_i \cdot \theta_{i-1}^{-1}$; notice that there is then a canonical isomorphism $f_n \simeq (p_0 \, \varphi_n)^* \, (f_0)$ while given the $\theta_{i-1 , i}$, we can set $f_n$ equal to $(p_0 \, \varphi_n)^* \, (f_0)$.
\end{proof}

\begin{defn}
We say that $f. : X. \to {\mathfrak X}$ is a {\rm proper hypercover} if
\begin{itemize}
\item[{\rm 0)}] $f_0 : X_0 \to {\mathfrak X}$ is proper and surjective.
\item[{\rm 1)}] For all $n \geq 1$, the map
$$
\alpha (f_n) = X_n \to \alpha ({\rm cosk}_{n-1} (f.)_n)
$$
is proper and surjective ({\rm i.e.} the natural map of simplicial schemes $X. \to {\rm cosk}_0^{\mathfrak X} (X_0).$ is a proper hypercover).
\end{itemize}
\end{defn}

Given any stack ${\mathfrak X}$ there exists a proper hypercover $X. \to {\mathfrak X}$ as above. Indeed, Chow's lemma implies the existence of $f_0 : X_0 \to {\mathfrak X}$ proper and surjective. Then choose $X. = {\rm cosk}_0^{\mathfrak X} (X_0)$.

\begin{lemma}
\label{proper2}
If ${\mathfrak X}$ is a stack and $f. : X. \to {\mathfrak X}$ is a
morphism from a simplicial scheme to ${\mathfrak X}$,
then for all morphisms $g : T \to {\mathfrak X}$ with $T$ a scheme,
$f.$ induces a map $f_T. : X. \times_{\mathfrak X} T \to T$ which is a proper
hypercover if $f.$ is. Furthermore, if $g : T \to {\mathfrak X}$ is \'etale and surjective,
then $f$ is a proper hypercover if and only if $f_T.$ is.
\end{lemma}

\begin{proof} First we need a technical lemma about fibre products.

\begin{lemma}
\label{fiber}
Let $f : X \to {\mathfrak X}$, $g : Y \to {\mathfrak X}$, $h: Z \to {\mathfrak X}$ be three morphisms from schemes to a stack ${\mathfrak X}$. Suppose we are given a map of schemes $p : X \to Y$, and a map $\tilde p : f \to g$ in ${\mathfrak X}$ covering $p$. Then the commutative square
$$
\xymatrix{
\alpha (f \times h) = X \times_{\mathfrak X} Z \ar[d] \ar[r] & Y \times_{\mathfrak X} Z = \alpha (g \times h) \ar[d] \\
\alpha (f = \tilde p \cdot g) = X \ar[r]^p   & Y = \alpha (g) \\
}
$$
is cartesian. {\rm I.e.} there is a canonical isomorphism
$$
X \times_{\mathfrak X} Z \overset{\sim}{\longrightarrow} X \times_Y Y \times_{\mathfrak X} Z \, .
$$
\end{lemma}

\begin{proof} Straightforward using the explicit description of fiber products in \cite{Laumon-M-B-stacks}, 2.2.2.
\end{proof}

Suppose now that $f.$ is a proper hypercover. Then, for $n=0$,
 $f_0 : X_0 \to {\mathfrak X}$ is proper and surjective.
Therefore $f_0 \times_{\mathfrak X} \, 1_T : X_0 \times_{\mathfrak X} T \to T$ is proper and surjective, by the definition of properness (\cite{Laumon-M-B-stacks}, Definition (7.11)).

\noindent For $n \geq 1$, we know that
$$
\alpha (f_n) = X_n \to \alpha ({\rm Cosk}_{n-1} (f.)_n)
$$
is proper and surjective. Now since ${\rm Cosk}_{n-1}$ commutes with products (since it is a right adjoint),
$$
{\rm cosk}_{n-1} (f. \times g)_n = {\rm cosk}_{n-1} (f.)_n \times g
$$
and so
$$
\alpha (f_n \times g) \to \alpha ({\rm cosk}_{n-1} (f. \times g)_n)
$$
is equal to
$$
X_n \times _{\mathfrak X} T \to \alpha ({\rm Cosk}_{n-1} (f.)_n \times g) \simeq \alpha ({\rm Cosk}_{n-1} (f.)) \times_{\mathfrak X} T
$$
which is again proper and surjective.

Suppose now that $g : T \to {\mathfrak X}$ is \'etale and surjective, and that $X. \times_{\mathfrak X} T \to T$ is a proper hypercover. Then
$$
X_0 \times_{\mathfrak X} T \to T
$$
is proper and surjective, and hence by \cite{Laumon-M-B-stacks} Remark (7.11.1) $X_0 \to {\mathfrak X}$ is proper and surjective. Similarly, for all $n \geq 1$,
$$
X_n \to \alpha ({\rm Cosk}_{n-1} (f.)_n)
$$
is proper and surjective. This proves Lemma~\ref{proper2}.\end{proof}
\begin{lemma}
\label{base change hypercovers stacks}
If $\mathfrak{X}$ is a stack, and $f. : X. \to \mathfrak{X}$ is a proper hypercover of
$\mathfrak{X}$ by a simplicial variety, then for all morphisms $g : Y. \to \mathfrak{X}$ with $Y.$ a simplicial variety, the induced map $f_{Y.} : Y. \times_{\mathfrak X} X. \to Y.$ is a proper hypercover.
\end{lemma}

\begin{proof} We must show that for all $n \geq 0$ the natural map
$$
Y_n \times_{\mathfrak X} X_n \to {\rm Cosk}_{n-1}^{Y.} (Y. \times_{\mathfrak X} X.)_n \leqno (*)
$$
is proper and surjective. It is straightforward to verify, knowing the adjunction between ${\rm Sk}_{n-1}$ and ${\rm Cosk}_{n-1}$, that there is a canonical isomorphism of schemes over $Y_n$:
$$
{\rm Cosk}_{n-1}^{Y.} (Y. \times_{\mathfrak X} X.)_n \to Y_n \times_{\mathfrak X} {\rm Cosk}_{n-1}^{\mathfrak X} (X.)_n
$$
and that the composition of these two maps is the base change, by $Y_n \to {\mathfrak X}$, of the proper surjective map $X_n \to {\rm Cosk}_{n-1}^{\mathfrak X} (X.)_n$. But  by Lemma~\ref{fiber}\textbf{},
$$
Y_n \times_{\mathfrak X} X_n \simeq (Y_n \times_{\mathfrak X} {\rm Cosk}_{n-1}^{\mathfrak X} (X.)_n) \times_{{\rm Cosk}_{n-1}^{\mathfrak X} (X.)_n} X_n
$$
and so the map (*) is proper and surjective by base change with respect to the map
$$
Y_n \times_{\mathfrak X} {\rm cosk}_{n-1}^{\mathfrak X} (X.)_n \to {\rm cosk}_{n-1}^{\mathfrak X} (X.)_n \, .
$$
\end{proof}

\subsection{Compactification of simplicial varieties}

Recall that for $n \geq 0$, an $n$-truncated simplicial variety $X.$ is said to be {\it split} if for all $k \leq n$, the complement in $X_k$ of the image of the degeneracies:
$$
NX_k = X_k \setminus \bigcup_{s_i : X_{n-1} \to X_n} s_i (X_{n-1})
$$
is both open and closed in $X_n$. It then follows that for $0 \leq k \leq n$:
$$
X_k = \coprod_{\alpha : [k] \twoheadrightarrow [\ell]} X(\alpha) (NX_{\ell}) \, .
$$
If $X.$ is a simplicial variety, we say that it is $n$-{\it split} if ${\rm sk}_n (X.)$ is split, and {\it split} if it is $n$-split for all $n \geq 0$.

\begin{prop}
\label{split}
Given a simplicial variety $X.$ with proper face maps, there is a proper hypercover $f. : Y. \to X.$ with $Y.$ split.
\end{prop}

\begin{proof} We shall construct, by induction on $n \geq 0$, simplicial varieties $Y(n).$ (with proper face maps), together with maps
$$
f_{n+1} : Y(n+1). \to Y(n).
$$
with the following properties:
\begin{itemize}
\item[1)] $Y(n).$ is $n$-split for all $n$
\item[2)] $Y(0). = X.$ (note that {\it any} simplicial variety is $0$-split)
\item[3)] For all $n \geq 1$, $f_n$ is a proper hypercover
\item[4)] For all $n \geq 1$, ${\rm sk}_{n-1} (f_n)$ is an isomorphism
\item[5)] The natural map
$$
Y(n). \to {\rm Cosk}_n^{X.} (Y(n).).
$$
is an isomorphism. (Notice that this is trivially true for $n=0$.)
\end{itemize}
Having constructed such a tower of maps, set (for $n \geq 1$) $\varphi_n : Y(n) \to Y(0) = X.$ equal to $f_1 \circ \ldots \circ f_n$, and then define $Y. = \underset{n}{\varprojlim} \, Y(n)$, and $\varphi : Y. \to X.$ equal to $\underset{n}{\varprojlim} \, \varphi_n$. Notice that the natural map $Y. \to Y. (n)$ induces an isomorphism ${\rm sk}_n (Y.) \to {\rm sk}_n (Y(n).)$, and hence $\varphi : Y. \to X.$ is a proper hypercover with $Y.$ split. The induction has already been started by setting $Y(0).$ equal to $X.$

Now suppose that $n \geq 1$ and that $Y(k).$ and $f_k$ have been defined for $k < n$ (if $n=1$, then we only need the existence of $Y(0). = X.$), satisfying conditions 1)--5) above. We start by setting
$$
{\rm sk}_{n-1} \, Y(n). = {\rm sk}_{n-1} \, Y(n-1) \, .
$$

Following \cite{SGA4-2} 5.1.3 we know that, given a variety $N$ and a map
$$
\beta : N \to {\rm Cosk}_{n-1} ({\rm Sk}_{n-1} (Y(n)))_n
$$
there is, up to isomorphism, a unique split $n$-truncated variety $V.$ with $NV_n = N$, ${\rm sk}_{n-1} (V.) = {\rm sk}_{n-1} (Y(n).) = {\rm sk}_{n-1} (Y(n-1))$, and $\beta$ the restriction of the natural map $V_n \to {\rm cosk}_{n-1} ({\rm sk}_{n-1} (V.))$ to $NV_n$.

Thus constructing the split $n$-truncated simplicial object ${\rm sk}_n (Y(n))$ is equivalent to giving a variety $V_n$, and a map $V_n \to {\rm Cosk}_{n-1} ({\rm Sk}_{n-1} \, Y(n-1))_n$. By \cite{Deligne-Hodge-III} Proposition 6.2.4, to give a map
$$
{\rm sk}_n (f_n) : {\rm sk}_n (Y(n)) \to {\rm sk}_n (Y(n-1))
$$
is the same as giving morphisms:
$$
{\rm sk}_{n-1} (f_n) : {\rm sk}_{n-1} (Y(n)) \to {\rm sk}_{n-1} (Y(n-1))
$$
and
$$
\tilde f : N \to Y(n-1)_n = {\rm Cosk}_{n-1}^{X.} (Y(n-1))_n
$$
such that the diagram:
$$
\xymatrix{
N \ar[d]^{\tilde f} \ar[r]^{\beta} & {\rm Cosk}_{n-1} (Y(n))_n \ar@{=}[d] \\
{\rm Cosk}_{n-1}^{X.} (Y(n-1))_n \ar@{=}[d] \ar[r]^{\pi}   & {\rm Cosk}_{n-1} (Y(n-1))_n \\
{\rm Cosk}_{n-1} (Y(n-1))_n \times_{{\rm Cosk}_{n-1} (X.)_n} \, X_n
}
$$
commutes. Here $\pi$ is the projection onto the first factor in the fibre product.
Hence the map $\tilde f$ {\it determines} $\beta$, the split object ${\rm sk}_n (Y(n))$, and the map ${\rm sk}_n (f_n)$.

Let us choose $N$ and $\tilde f$ so that $\tilde f$ is proper and surjective; for example, take $\tilde f$ to be the identity. Having defined ${\rm sk}_n (Y(n).)$ we now set $Y(n). = {\rm cosk}_n^{X.} ({\rm sk}_n (Y(n).))$.

To define $f_n$, observe that, by \ref{props-cosk},
$$
Y(n-1). = {\rm Cosk}_{n-1}^{X.} (Y(n-1).) = {\rm Cosk}_n^X (Y(n-1)).
$$
and so we set $f_n = {\rm cosk}_n^{X.} ({\rm sk}_n (f_n))$.

It remains to show that $f_n$ is a proper hypercover. Since $(f_n)_p$ is an isomorphism for $p < n$, we know that
$$
Y(n)_p \to {\rm Cosk}_{p-1}^{Y(n-1)} \, (Y(n))_p
$$
is an isomorphism, hence proper and surjective for $p < n$. Now suppose $p \geq n$.

First observe that
\begin{multline*}
 {\rm Cosk}_{n-1}^{Y(n-1)} (Y(n))_n = \\
{\rm Cosk}_{n-1} (Y(n))_n \times_{{\rm Cosk}_{n-1} (Y(n-1))_n} Y(n-1)_n = Y(n-1)_n
\end{multline*}
since ${\rm sk}_{n-1} Y(n) \overset{\sim}{\longrightarrow} {\rm sk}_{n-1} Y(n-1)$, and so
$${\rm Cosk}_{n-1} (Y(n)) \simeq {\rm Cosk}_{n-1} (Y(n-1))\; .$$
However, $f(n)_n : Y(n)_n \to Y(n-1)_n$ is proper and surjective,
since $\beta$ is proper and surjective,
and the face (and hence also the degeneracy) maps of $Y(n-1)$ are proper by the induction hypothesis. Hence ${\rm sk}_n \, Y(n) \to {\rm sk}_n (Y(n-1))$ is a proper hypercover of $n$-truncated varieties (\ref{fiber}). We then conclude with the following lemma:

\begin{lemma}
\label{T}
Let $Y.$ and $Z.$ be $n$-truncated simplicial varieties over the simplicial variety $X.$, and suppose we are given a proper hypercover $f. : Y. \to Z.$ (over $X.$). Then
$$
{\rm cosk}_n^X (Y). \to {\rm cosk}_n^X (Z).
$$
is a proper hypercover.
\end{lemma}

\begin{proof} We have to prove the lifting property for all inclusions of finite simplicial sets $i : A \hookrightarrow B$. Suppose then that we are given a commutative diagram, with $S = {\rm Spec} (F)$ for $F$ an algebraically closed field:
$$
\xymatrix{
A \times S \ar@{^{(}->}[d] \ar[r] & {\rm Cosk}_n^X (Y.) \ar[d]  \\
B \times S \ar[r] \ar@{-->}[ur]^{\tilde\theta} &{\rm Cosk}_n^X (Z.) \\
} \leqno (*)
$$
This is then equivalent to giving a commutative diagram
$$
\xymatrix{
{\rm sk}_n (A) \times S \ar[d] \ar[r] & Y. \ar[d]  \\
{\rm sk}_n (B) \times S \ar[d] \ar[r] \ar@{-->}[ur]^{\theta} &Z. \ar[d] \\
B \times S \ar[r] &X.
}
$$
and hence a lifting $\theta$ exists, since $Y. \to Z.$ is a proper hypercover, and $\theta$ defines a lifting $\tilde\theta$ in diagram (*). To see that, for every valuation ring $\Lambda$, with fraction field $F$, a lifting $\tilde\psi$ exists in every diagram, with $i : A \hookrightarrow B$ as above,
$$
\xymatrix{
A \times {\rm Spec} (\Lambda) \cup B \times {\rm Spec} (F) \ar@{^{(}->}[d] \ar[r] & {\rm Cosk}_n^X (Y.) \ar[d]  \\
B \times {\rm Spec} (\Lambda) \ar[r] \ar@{-->}[ur]^{\tilde\psi} &{\rm Cosk}_n^X (Z.) \\
}
$$
it is enough to produce a lifting $\psi$ in the diagram:
$$
\xymatrix{
{\rm sk}_n (A) \times {\rm Spec} (\Lambda) \cup {\rm sk}_n \, B \times {\rm Spec} (F) \ar@{^{(}->}[d] \ar[r] & Y. \ar[d]  \\
{\rm sk}_n (B) \times {\rm Spec} (\Lambda) \ar[d] \ar[r] \ar@{-->}[ur]^{\psi} &Z. \ar[d] \\
B \times {\rm Spec} (\Lambda) \ar[r] &X.
}
$$
which again exists since $Y. \to Z.$ is a proper hypercover.\end{proof}

\begin{prop}
\label{compact}
Let $X.$ be a split simplicial variety with proper face maps. Then there is a map of simplicial varieties $j. : X. \to \bar X.$ such that:
\begin{itemize}
\item[{\rm 1)}] For all $k \geq 0$, $j_k : X_k \hookrightarrow \bar X_k$ is an open immersion with $\bar X_k$ proper over $S$
\item[{\rm 2)}] $\bar X.$ is split.
\end{itemize}
\end{prop}

\begin{proof} We construct ${\rm sk}_n (\bar X.)$ by induction on $n \geq 0$. For $n=0$, we set $\bar X_0$ equal to a ``compactification'' of $X_0$ over $S$, which exists by Nagata's theorem.

Suppose now that $n \geq 1$, and that we have constructed ${\rm sk}_{n-1} (\bar X.)$, together with a map
$$
{\rm sk}_{n-1} (j.) : {\rm sk}_{n-1} (X.) \to {\rm sk}_{n-1} (\bar X.)
$$
satisfying 1) above.

Giving ${\rm sk}_n (\bar X.)$ is the same as giving $N \bar X_n$ and a map
$\bar\beta_n : N\bar X_n \to {\rm cosk}_{n-1} ({\rm sk}_{n-1} (\bar X.))_n$.
Since $X.$ is split, we have maps
$$
\beta_n : NX_n \to {\rm cosk}_{n-1} ({\rm sk}_{n-1} (X.))_n
$$
and
$$
{\rm cosk}_{n-1} ({\rm sk}_{n-1} (j.)) : {\rm cosk}_{n-1} ({\rm sk}_{n-1} (X.)) \to {\rm cosk}_{n-1} ({\rm sk}_{n-1} (\bar X.)) \, .
$$
By Nagata's theorem, there exists a factorization of ${\rm cosk}_{n-1} ({\rm sk}_{n-1} (j.)) \circ \beta_n$:
$$
\xymatrix{
NX_n \ar[d]^{\beta_n} \ar[rr]^{i_n} &&\overline{NX_n} \ar[d]^{\bar\beta_n}  \\
{\rm cosk}_{n-1} ({\rm sk}_{n-1} (X.)) \quad \ar[rr]^{{\rm cosk}_{n-1} ({\rm sk}_{n-1} (j)) } &&\quad {\rm cosk}_{n-1}({\rm sk}_{n-1} (\bar X.)) \\
}
$$
with $\bar\beta_n$ proper, and $i_n$ an open immersion.

By \cite{SGA4-2} there is then a unique  split $n$-truncated  simplicial variety ${\rm sk}_n (\bar X.)$ with ${\rm sk}_{n-1} ({\rm sk}_n (\bar X.)) = {\rm sk}_{n-1} (\bar X.)$ and $N\bar X_n$ equal to $\overline{NX_n}$. Then, following \cite{Deligne-Hodge-III}, the composition of $i_n$ with the inclusion $N \bar X_n \hookrightarrow \bar X_n$, induces a map
$$
{\rm sk}_n (j) : {\rm sk}_n (X.) \to {\rm sk}_n (\bar X.)
$$
with $j_n$ an open immersion, since it is a disjoint union of open immersions. This completes the induction step.\end{proof}

\section{Homological descent}

In the following, $E$ will be a covariant functor from the category of proper morphisms between schemes to the category of connective spectra.  This includes the case of functors to the category of
chain complexes of abelian groups concentrated in degrees $\geq 0$.

We assume in addition that $E$ is contravariant with respect to open immersions.
We suppose that the two types of functoriality are related as follows:
\begin{itemize}
\item[1.] If $j : U \subset X$ is an open subset with complement $Y = X \backslash U$, then:
$$
E(Y) \to E(X) \to E(U)
$$
is a (co-)fibration sequence.
\item[2.] If $j : U \subset X$ is an open subset, and $f : Z \to X$ is proper, the diagram:
$$
\xymatrix{
E(Z)\ar[d]_{f_*}\ar[r]^{\!\!\!\!\!\!\!\!\!\!(j\vert_Z)^*} &E(f^{-1} (U)) \ar[d]^{(f\vert_U)_*}  \\
E(X)\ar[r]_{j^*}  &E(U) \\
}
$$
commutes.
\end{itemize}
Properties 1. and 2. imply that if $X$ and $Y$ are schemes then the inclusion maps of $X$ and $Y$ into the disjoint union $X\sqcup Y$ induce a weak equivalence 
$$E(X) \vee E(Y) \to E(X \cup Y).$$
We also assume that $E$ satisfies devissage:
\begin{itemize}
\item[3.] For all $X$, the natural map $j : X^{\rm red} \to X$ from the reduced structure on $X$ to $X$ induces a weak equivalence on $E$.
\end{itemize}

In this section we will prove that a functor $E$ having properties (1.)to (3.) above
satisfies descent with respect to hyperenvelopes. 

If in addition $E$ is contravariant with respect to finite flat morphisms, and satisfies properties (4.) and (5.) below, we shall show that $E$, with rational coefficients, satisfies descent with respect to all proper hypercovers.
\begin{itemize}
\item[4.] The diagram analogous to that in item (2.) above, in which $U\to X$ is finite and flat, is also strictly commutative.
\item[5.] If $p : X \to Y$ is finite and flat, with $p_* (\mathcal{O}_X)$ a {\it free} $\mathcal{O}_Y$-module of rank $n$, then
$$
p_* \, p^* : E_* (Y) \to E_* (Y)
$$
is multiplication by $n$.
\end{itemize}

The $K$-theory of coherent sheaves ($G$-theory), or the homology
theory corresponding to a cycle complex in the sense of Rost,
\cite{Rost-Chow-groups-with-coefficients}, and in particular the
homology theory associated to the Gersten complexes
\cite{Gillet-RR}, all satisfy properties (1.) to (5.).

In the paper \cite{Gillet-homological-descent}, descent for $K'$-theory with respect to hyperenvelopes was proved first for the homology of the Gersten complexes and then using a spectral sequence argument, for $K'$-theory.  Here we shall prove descent for the map on homology with rational coefficients induced by an arbitrary proper hypercover of simplicial schemes, following the method of SGA~4 \cite{SGA4-2}.

We start by extending $E$ from varieties to simplicial varieties:

\begin{defn}
If $X.$ is a simplicial variety with proper face maps, and $E$ is any covariant functor from the category of proper morphisms between schemes to the category of connective spectra, then we define:
$$
E(X.) := {\rm hocolim}_{i \in \Delta^{\rm op}} \, E(X_i) \, .
$$
This is clearly covariant functorial with respect to proper morphisms between simplicial schemes with proper face maps.
\end{defn}

\begin{lemma}
\label{DescentSpectralSequence}
Let $X.$ be a simplicial scheme with proper face maps, and suppose that $E$ is any covariant functor from the category of proper morphisms between schemes to the category of connective spectra. Then there is a (convergent) first quadrant homological spectral sequence:
$$
E_{p,q}^1 = E_p (X_q) \Rightarrow E_{p+q} (X.) \, .
$$
\end{lemma}

\begin{proof} This is the standard spectral sequence for hocolim \cite{BK}. \end{proof}

\begin{lemma} Suppose that $E$ is a covariant functor from the category of proper morphisms between schemes to the category of connective spectra which takes disjoint unions of schemes to direct sums of spectra  (in particular if $E$ satisfies properties (1.) and (2.)).  Then if $f,g : X. \to Y.$ are homotopic maps in the category of proper morphisms between simplicial schemes with proper face maps, $E(f) , E(g) : E(X.) \to E(Y.)$ are equal in the homotopy category of spectra.
\end{lemma}

\begin{proof} There is a map $h : X. \times \Delta^1 \to Y.$ such that $h |_{X.\times{0}} = f$ and $h |_{X.\times{1}} = g$.  It suffices to observe that the inclusions $X. \times \{ i \} \to X. \times \Delta^1$, for
$i = 0,1$, induce weak equivalences of $E$-theory spectra inverse to the weak equivalence induced by the projection $X. \times \Delta^1 \to X.$.  By Lemma~\ref{DescentSpectralSequence}, it is enough to
check that these maps  induce isomorphisms on the chain complexes obtained by applying the functor $E_q$ for all $q$.  But $E_q$ takes disjoint unions of schemes to direct sums of abelian
groups, and for any functor $H$ from schemes to abelian groups with this property, the maps of simplicial abelian groups (for $i = 0,1$):
$$
H (X. \times \{ i \}) \to H(X. \times \Delta^1) \simeq H(X.) \otimes \mathbf{Z} (\Delta^1) \to H(X.)
$$
are weak equivalences. \end{proof}

\begin{thm}
\label{descent}  Suppose that $E$ is a covariant functor from the category of proper morphisms between schemes to the category of connective spectra satisfying properties (1.) to (5.).
Let $f. : X. \to Y.$ be a proper hypercover between simplicial $S$-varieties, each having proper face maps. Then $E (f.)_\mathbb{Q} : E(X.)_\mathbb{Q} \to E(Y.)_\mathbb{Q}$ is a weak equivalence.
\end{thm}

There are several steps to the proof.

\begin{lemma}
\label{Lemme308}
Let $p : X \to Y$ be a proper morphism that admits a section $s: Y \to X$. Then for any covariant 
functor $E$ from the category of proper morphisms between schemes to the category of connective spectra which takes a disjoint union of schemes to a direct sum of spectra, in particular for any $E$ which satisfies properties (1.) and (2.) above, we have that:
$$
{\rm cosk}_0^Y (p) : E ({\rm cosk}_0^Y (X).) \to E (Y)
$$
is a weak equivalence.
\end{lemma}

\begin{proof} Let $Y.$ be the constant simplicial object $[n] \mapsto Y$.
Then $s$ induces a section $\tilde s : Y. \to {\rm cosk}_0^Y (X).$ of the projection ${\rm cosk}_0^Y (p)$.
On the other hand, by Lemma~\ref{Lemme212}, $ \tilde s \circ {\rm cosk}_0^Y (p) = {\rm cosk}_0^Y (s \circ p)$ is
homotopic to the identity and so we are done.
(Note that here we do not need to assume that $E$ has rational coefficients.) \end{proof}

\begin{lemma}
\label{Lemme309}
Let $p : X \to Y$ be a finite flat morphism, such that $p_* (\mathcal{O}_X) \simeq \mathcal{O}^d_Y$ for some $d \geq 0$. (Note that if $p$ is finite and flat  this will be true on a Zariski dense open subset of $Y$.) Then
$$
E ({\rm cosk}_0^Y(X).)_\mathbb{Q} \to E (Y)_\mathbb{Q}
$$
is a weak equivalence.
\end{lemma}

\begin{proof} Following Lemma~\ref{DescentSpectralSequence}, it is enough to prove that, for all
$q \geq {0}$,
$$
H_p (i \mapsto E_q ({\rm cosk}_i^Y (X))_\mathbb{Q}) \simeq \left\{
 \begin{array}{ll}
0, & \hbox{$i>0$;} \\
E_q(Y)_\mathbb{Q}, & \hbox{$i=0$.}
\end{array}
\right.
$$
Consider, therefore, the corresponding augmented chain complex $C_*$ concentrated in degrees $\geq -1$, with $C_n = E_q ({\rm cosk}^Y (X)_n)_\mathbb{Q}$ for $n \geq 0$,  $C_{-1} = E_q (Y)_\mathbb{Q}$, and differentials $\delta_n = \sum_{i=0}^n (-1)^i (d_i)_*$ for $n \geq 0$, where we set $d_0 = p : X = {\rm cosk}^Y (X)_0 \to {\rm cosk}^Y (X)_{-1} = Y$ when $n = 0$.

For all $n \geq 0$, and all $i \leq n$, one may easily check that the following square is cartesian, with all maps finite and flat :
$$
\begin{CD}
{\rm cosk}^Y (X)_{n+1} @>d_{n+1}>> {\rm cosk}^Y (X)_n \\
@Vd_iVV                      @VVd_iV \\
{\rm cosk}^Y (X)_n @>d_n>> {\rm cosk}^Y (X)_{n-1}
\end{CD}
$$
and hence that 
$$(d_n)^* (d_i)_* = (d_i)_* (d_{n+1})^* : E_q ({\rm cosk}^Y (X)_n) \to E_q({\rm cosk}^Y (X)_n).$$  
It is then straightforward to check that the sequence of maps, for $n \geq -1$ :
$$
h_n = \frac{(-1)^{(n+1)}}{d} (d_{n+1})^* : E_q ({\rm cosk}^Y (X)_n)_\mathbb{Q} \to E_q ({\rm cosk}^Y (X)_{n+1})_\mathbb{Q}
$$
is a contracting homotopy for the complex $C_*$, given that for all $n \geq -1$, $(d_{n+1})_* (d_{n+1})^* : E_q({\rm cosk}^Y (X)_n) \to E_q({\rm cosk}^Y (X)_n)$ is (by the projection formula) multiplication by $d$. \end{proof}

\begin{prop}
\label{proposition3010}
Let $p : X \to Y$ be a proper surjective morphism.  Then
$$
E ({\rm cosk}^Y (X).)_\mathbb{Q} \to E (Y)_\mathbb{Q}
$$
is a weak equivalence.
\end{prop}

\begin{proof} By devissage we may assume that $Y$ is reduced. We proceed by noetherian induction on the closed subsets of $Y$. The assertion is trivially true if $Y$ is empty. Now suppose that $Y$ is non-empty, and let $\eta \in Y$ be a generic point of a component of $Y$.  Then there exists a point $\xi \in X_\eta$ (the fiber over the point $\eta$) such that the residue field $\mathbf{k} (\xi)$ is a
finite extension of $\mathbf{k} (\eta)$.  Taking the Zariski closure $\bar{\xi}$ of $\xi$ in $X$, we obtain a subscheme of $X$ which is finite over an open neighborhood of $\eta$. By generic flatness, there is then an open neighborhood $U$ of $\eta$, the inverse image $V$ in $\bar{\xi}$ of which is finite and flat over $U$. Additionally, we may assume that  $p_* (\mathcal{O}_V) \simeq \mathcal{O}^d_U$.

By the induction hypothesis the assertion of the proposition is true for $p_Z : Z \times_Y X \to Z$, where $Z = Y \backslash U$.

Now consider the diagram:
$$
\begin{CD}
E ({\rm cosk}^Z (p^{-1} (Z)).)_{\mathbb Q} @>>> E ({\rm cosk}^Y (X).)_{\mathbb Q} @>>> E({\rm Cosk}^{p^{-1} (U)} (U).)_{\mathbb Q} \\
@VfVV  @VgVV @VhVV \\
E(Z)_{\mathbb Q} @>>> E(X)_{\mathbb Q} @>>> E(U)_{\mathbb Q} \, .
\end{CD}
$$
The map $f$ is a weak equivalence by the induction hypothesis. Next, observe that:
$$
E(Z)_{\mathbb Q} \to E(X)_{\mathbb Q} \to E(U)_{\mathbb Q}
$$
is a fibration sequence by localization, and similarly, for all $n \geq 0$,
$$
{\rm cosk}^Z (p^{-1} (Z))_n = {\rm cosk}^Y (X)_n \backslash {\rm cosk}^{p^{-1} (U)} (U)_n
$$
and so, again by localization, we have a fibration sequence of connective spectra
$$
E ({\rm cosk}^Z (p^{-1} (Z))_n) \to E ({\rm cosk}^Y (X)_n) \to E ({\rm cosk}^{p^{-1} (U)} (U)_n) \, .
$$
Hence, since by \cite{BK}, hocolim preserves fibration sequences of connective spectra, we have a fibration sequence:
$$
E ({\rm cosk}^{p^{-1} (Z)} (Z).)_{\mathbb Q} \to E ({\rm cosk}^Y (X).)_{\mathbb Q} \to E ({\rm cosk}^{p^{-1} (U)} (U).)_{\mathbb Q} \, .
$$
Hence, in order to prove that $E ({\rm cosk}^Y (X).)_{\mathbb Q} \to E (Y.)_{\mathbb Q}$ is a weak equivalence, it remains to show (writing $Y$ for $U$ and $X$ for $p^{-1} (U)$) that if $p : X \to Y$ is a proper surjective morphism such that there is a finite flat map $\pi : V \to Y$, and a section $s : V \to X$, {\it i.e.} such that $p \circ s = \pi$, then $p_* : E ({\rm cosk}^{Y} \cdot (X))_\mathbb{Q} \to E (Y)_\mathbb{Q}$ is a weak equivalence.

Consider the bisimplicial scheme ${\rm cosk}^Y (V). \times_Y {\rm cosk}^Y (X).$.
For each $n \geq 0$, $E ({\rm cosk}^Y (V). \times_Y {\rm cosk}^Y (X)_n)_\mathbb{Q} \to E ({\rm cosk}^Y (X)_n)_\mathbb{Q}$ is a weak equivalence by Lemma~\ref{Lemme308}.  Hence (by the homotopy colimit theorem \cite{BK}) $E ({\rm cosk}^Y (V). \times_Y {\rm cosk}^Y (X).)_\mathbb{Q} \to
E ({\rm cosk}^Y (X).)_\mathbb{Q}$ is too.

On the other hand, for each $m\geq 0$,
$$
{\rm cosk}^Y_m (V) \times_Y {\rm cosk}^Y (X). \simeq {\rm cosk}^{{\rm cosk}^Y_m(V)} (X \times_Y{\rm cosk}^Y (V)_m).
 $$
and since there is a section $V \to X$ over $Y$, there is a section
${\rm cosk}^Y_m (V) = \prod_Y^m (V) \to X \times_Y {\rm cosk}^Y_m (V)$.
Therefore, the augmentation
$$E ({\rm cosk}^Y (V)_m \times_Y{\rm cosk}^Y (X).) \to E({\rm cosk}^Y (V)_m)$$
is a weak equivalence
(even without tensoring with $\mathbb{Q}$), by \ref{Lemme308},
and so again by the homotopy colimit theorem \cite{BK}, we have a weak equivalence:
$E({\rm cosk}^Y (V). \times_Y {\rm cosk}^Y (X).) \to E({\rm cosk}^Y (V).)$.

Finally, by Lemma~\ref{Lemme309}, we know that $E ({\rm cosk}^Y (V).)_\mathbb{Q} \to
E(Y)_\mathbb{Q}$ is a weak equivalence, and it follows therefore that  $E({\rm cosk}^Y (X).)_\mathbb{Q} \to E(Y)_\mathbb{Q}$ is \break too. \end{proof}

\begin{proof}[Proof of Theorem~{\rm \ref{descent}}] Suppose that $f. : X. \to Y.$ is a proper hypercover of simplicial varieties with proper face maps. By \ref{props-cosk}, we know that the sequence of simplicial varieties:
$$
\ldots \to {\rm cosk}_n^{Y.} (X.) \to {\rm cosk}^{Y.}_{n-1}(X.)\to\ldots Y.
$$
satisfies properties 1)--3) of Lemma~\ref{props-tower}. \end{proof}

Since $E_q = 0$ for $q < 0$, it follows from the descent spectral sequence \ref{DescentSpectralSequence}, that if $p. : V. \to W.$ is a proper map between simplicial varieties with proper face maps, and there is an $n \geq 1$ such that $p_i$ is an isomorphism for $i \leq n$, then $E_q (V.) \to E_q (W.)$ is an isomorphism for $q < n$. It follows that for a given $q \geq 0$, the map
$$
E_q (X.) \to E_q ({\rm cosk}_n^{Y.} (X.).)
$$
is an isomorphism once $n > q$. It will therefore be sufficent to show, for all $n \geq 0$, that
$$
E({\rm cosk}_n^{Y.} (X.).)_{\mathbb Q} \to E ({\rm cosk}_{n-1}^{Y.} (X.).)_{\mathbb Q}
$$
is weak equivalence. First we need the following lemma:

\begin{lemma}
Let $f : X. \to Y.$ be a proper hypercover. Then for all $n \geq 0$ and all $p \geq 0$,
$$
(f_n)_p : {\rm Cosk}_n^{Y.} (X.)_p \to {\rm Cosk}_{n-1}^{Y.} (X.)_p
$$
is proper and surjective.
\end{lemma}

\begin{proof} We have to show that $(f_n)_p$ is surjective on $F$-valued points for all algebraically closed fields, and that the valuative criterion of properness holds. We shall only prove the valuative criterion, since the proof of surjectivity is similar, but easier.

Suppose then that we are given a commutative square, for $\Lambda$ a valuation ring, with fraction field $F$ :
$$
\xymatrix{
{\rm Spec} (F) \ar@{^{(}->}[d] \ar[r]^i & {\rm Cosk}_n^{Y.} (X.)_p \ar[d]^{(f_n)_p}  \\
{\rm Spec} (\Lambda) \ar[r] \ar@{-->}[ur]^{\tilde j} &{\rm Cosk}_{n-1}^{Y.} (X.)_p \\
} \, .
$$
Since
$$
{\rm Cosk}_n^{Y.} (X.)_p \cong {\rm Cosk}_n (X.)_p \times_{{\rm Cosk}_n (Y.)_p} Y_p
$$
the commutative square above is equivalent to a commutative diagram
$$
\xymatrix{
({\rm sk}_n (\Delta^p) \times {\rm Spec} (F)) \cup ({\rm sk}_{n-1} (\Delta^p) \times {\rm Spec} (\Lambda)) \ar@{^{(}->}[d] \ar[r] & X. \ar[d]^f  \\
{\rm sk}_n (\Delta^p) \times {\rm Spec} (\Lambda) \ar@{^{(}->}[d] \ar[r] \ar@{-->}[ur]^{\theta} &Y. \ar@{=}[d] \\
\Delta^p \times {\rm Spec} (\Lambda) \ar[r] &Y.
}
$$
and giving the lifting $\tilde j$ is equivalent to giving the map $\theta$, which exists since $f$ is a proper hypercover. \end{proof}

We now have that the map
$$
f_n : {\rm cosk}_n^{Y.} (X.). \to {\rm cosk}_{n-1}^{Y.} (X.).
$$
has the following properties (see \ref{props-cosk})
\begin{itemize}
\item[(1)] For all $p$, $(f_n)_p$ is proper and surjective.
\item[(2)] For $p < n$, $(f_n)_p$ is an isomorphism.
\item[(3)] For $p \geq n$, the natural map
$$
{\rm Cosk}_n^{Y.} (X) \to {\rm Cosk}_p^{Y.} ({\rm Cosk}_n^{Y.} (X))
$$
is an isomorphism.
\end{itemize}

Suppose now that $n \geq 0$ and that $f : V. \to W.$ is a map of simplicial varieties over $Y.$ with proper face maps satisfying (1)--(3) above. We claim that $E(V.)_{\mathbb Q} \to E(W.)_{\mathbb Q}$ is a weak equivalence. To see this, consider the bisimplicial scheme:
$$
([k] , [\ell]) \mapsto {\rm cosk}_0^{W_{\ell}} (V_{\ell})_k = V_{\ell} \times_{W_{\ell}} \times \ldots \times_{W_{\ell}} V_{\ell} \qquad \mbox{($k$ factors)}.
$$
By \cite{BK}, we have:
\begin{align}
&&\underset{\Delta^{\rm op} \times \Delta^p}{\rm hocolim} (([k] , [\ell]) \mapsto E({\rm cosk}_0^{W_{\ell}} (V_{\ell})_k)_{\mathbb Q}) \nonumber \\
&\simeq &\underset{\Delta^{\rm op}}{\rm hocolim} ([\ell] \mapsto \underset{\Delta^{\rm op}}{\rm hocolim} ([k] \mapsto E({\rm cosk}_0^{W_{\ell}} (V_{\ell})_k)_{\mathbb Q}) \nonumber \\
&\simeq &\underset{\Delta^{\rm op}}{\rm hocolim} ([k] \to \underset{\Delta^{\rm op}}{\rm hocolim} ([\ell] \mapsto E({\rm cosk}_0^{W_{\ell}} (V_{\ell})_k)_{\mathbb Q}) \, . \nonumber
\end{align}
Since $f_{\ell}$ is proper and surjective for all $\ell \geq 0$, we have that, for all $\ell \geq 0$ :
$$
E({\rm cosk}_0^{W_{\ell}} (V_{\ell}).)_{\mathbb Q} = \underset{\Delta^{\rm op}}{\rm hocolim} ([k] \mapsto E({\rm cosk}_0^{W_{\ell}} (V_{\ell}))_k)_{\mathbb Q} \to E (W_{\ell})_{\mathbb Q}
$$
is a weak equivalence (by Proposition~\ref{proposition3010}). Hence the natural map from the first iterated hocolim to $E(W.)_{\mathbb Q}$ is a weak equivalence.

On the other hand, consider, for a given $k \geq 1$, the face maps
$$
d_i^{W.} : {\rm cosk}_0^{W.} (V.)_k = W. \times_{V.} \times \ldots \times_{V.} W. \to {\rm cosk}_0^{W.} (V.)_{k-1} = W. \times_{V.} \times \ldots \times_{V.} W. \, .
$$
Each map $d_i^{W.}$ has a section, either $s_{i-1}^{W.}$ or $s_i^{W.}$. Also, since $V_p \to W_p$ is an isomorphism for $p < n$, the same is true for $d_i^{W_p}$. Furthermore, since $V. \to {\rm Cosk}_p^{W.} (V.)$ is an isomorphism for all $p \geq n$, it is straightforward to check that the same is true for ${\rm Cosk}_0^{W.} (V.)_k$ instead of $V.$ for all $k$. Thus, for $i < n$, the two maps $s_i^{W.} \, d_i^{W.} : {\rm Cosk}_0^{W.} (V.)_k \to {\rm Cosk}_0^{W.} (V.)_k$ and the identity are homotopic, by Lemma~\ref{existence-homotopies}. Similarly, if $i > 0$, the maps $s_{i-1}^{W.} \, d_i^{W.}$ and the identity are homotopic.

Applying the functor $E$, we see that $E(d_i^{W.})_{\mathbb Q}$ is an equivalence, independent of $i$, with inverse $E(s_i^{W.})_{\mathbb Q}$ and/or $E(s_{i-1}^W)_{\mathbb Q}$. It follows that the natural map
$$
V. \to {\rm Cosk}_0^{W.} (V.)..
$$
induces an equivalence on $E_{\mathbb Q}$, since for each $q \geq 0$,
$$
k \mapsto E_q (\underbrace{V. \times_{W.} \times \ldots \times_{W.} V.}_{\mbox{\footnotesize $k$-times}})
$$
is a constant simplicial group. Hence the map from $E(V.)_{\mathbb Q}$ to the second iterated hocolim above is also a weak equivalence, and the map
$$
E(V.)_{\mathbb Q} \to E(W.)_{\mathbb Q}
$$
is a weak equivalence. \end{proof}

\begin{thm}
\label{descent_hyperenvelope}
Suppose that $E$ is a covariant functor from proper morphisms between schemes to connective spectra which satisfies properties (1)-(3) but not necessarily properties (4) and (5). Then, if  $f. : X. \to Y.$ is a proper hyperenvelope between simplicial $S$-varieties, each having proper face maps, $E (f.) : E(X.) \to E(Y.)$ is a weak equivalence.
\end{thm}
\begin{proof} The proof of this result is essentially already in \cite{Gillet-homological-descent} and \cite{Gillet-Soule-motives-descent}, but for completeness we indicate how the proof differs from that of 
Theorem \ref{descent}.  The only difference is in the proof of the analog of Proposition \ref{proposition3010}, which states that if $p : X \to Y$ is an envelope, then
$$
E ({\rm cosk}^Y (X).) \to E (Y)
$$
is a weak equivalence. As in the proof of Proposition \ref{proposition3010},  using localization and noetherian induction on the closed subsets of $Y$ it is enough to know that for any $Y$ there is a non-empty Zariski open subset $U\subset Y$ such that 
$$
E ({\rm cosk}^U (p^{-1}(U)).) \to E (U)
$$
is a weak equivalence.  
However, since $p$ is an envelope, if $\eta$ is a  generic point of a component of $Y$ there is an open neighbourhood $U$ of $\eta$ such that the map $p_U:p^{-1}(U)\to U$ has a section, and hence the map above is a weak equivalence by Lemma \ref{Lemme308}.
\end{proof}

\section{Review of $K$-theory}

\subsection{The homology theory associated to $G$-theory}

As in \cite{Gillet-Soule-motives-descent}, we view the $K$-theory of coherent sheaves as a covariant
functor from the category of proper morphisms between schemes to the category of Waldhausen categories
(here we follow the terminology of \cite{Thomason_Trobaugh}, 1.2.3, rather than use the original terminology
``category with cofibrations and weak equivalences''), by sending $X$ to the category of bounded below complexes of
flasque sheaves of $\mathcal{O}_X$-modules which have bounded coherent cohomology;
see {\it op. cit.} 3.16. Indeed this is a strict functor, and composing
with the $K$-theory spectrum functor,
we get, as in \cite{Geisser-Hesselholt}, a covariant functor from the category of proper morphisms between
$S$-schemes to the category $\mbox{\textbf{Spectra}}$ of symmetric spectra of
\cite{Hovey-Shipley-Smith}. We then compose with the $\mathbb{Q}$-localization functor
(smash product with the Eilenbarg-Maclane spectrum $H\mathbb{Q}$)
$\mbox{\textbf{Spectra}} \to \mbox{\textbf{Spectra}}_\mathbb{Q}$, to obtain the
$K$-theory functor that we will use: $X. \mapsto \mathbf{G}(X.) \in \mbox{\textbf{Spectra}}_\mathbb{Q}$.
Note that we are following Thomason's notation ($G$-theory), rather than Quillen's ($K'$-theory).

Since we also want $G$-theory to be contravariant with respect to flat
(and in particular \'{e}tale) morphisms, it will be important to consider our complexes of sheaves to be
sheaves of $\mathcal{O}_X$-modules in the \emph{fppf} topology. Since being flasque is compatible with
flat base change, the resulting Waldhausen category is equivalent, and hence has the same $G$-theory,
as the category of complexes of sheaves in the Zarisksi topology.

We extend the $G$-theory functor first of all to the category of (degreewise)
proper morphisms between simplicial schemes with proper face maps by taking the homotopy
colimit of the associated simplicial spectra.
We may then extend one step further to the category $Ar (\mathbf{sP})$ with
objects proper morphisms $f. : Y. \to Z.$ between simplicial simplicial schemes
with proper face maps, and morphisms  $(g_Y, g_Z) : (f_1 : Y_1 \to Z_1) \to (f_2:Y_2 \to Z_2)$
pairs of proper maps such that $g_Z \circ f_1 = f_2\circ g_Y$, by setting $\mathbf{G} (f. : Y. \to Z.) := \mathrm{Cone} (\mathbf{G}(f.))$.

From theorem \ref{descent}, we have:
\begin{thm}\label{descent_for_Gthy}
Let $f. : X. \to Y.$ be a morphism between simplicial simplicial
schemes with proper face maps which is a  proper hypercover.  Then
$G(f.)$ is a weak equivalence.
\end{thm}

\subsection{$G$-theory of stacks}

\begin{defn}
If $\mathfrak{X}$ is a stack, we define the $G$-theory spectrum $\mathbf{G}(\mathfrak{X})$ to be $\mathbf{H}_{\text{\'et}}(\mathfrak{X},\mathbf{G})$,
where $\mathbf{G}$ is the presheaf of ($\mathbf{Q}$-localized) spectra on the \'etale site induced by the functor $\mathbf{G}$.
See {\rm \cite{Gillet-Intersection-theory-on-stacks}}, though here we are using the hypercohomology of (pre-)sheaves of spectra in the sense of
{\rm \cite{JardinePresheavessymmetricspectra}}.
\end{defn}

It follows from \'etale descent for rational $G$-theory (due to
Thomason, (Theorem 2.15 of \cite{AKTEC}) that this definition is
consistent with the definition for schemes.  In particular, we have:
\begin{lemma}
If $\pi:V.\to \mathfrak{X}$ is an \'{e}tale hypercover of a
Deligne-Mumford  stack, the natural map:
$$G(\mathfrak{X})\to \mathrm{holim}_i(G(V_i))$$
is a weak equivalence.
\end{lemma}

\begin{lemma}
\label{localization}
Let ${\mathfrak Y}$ be a closed substack of the stack ${\mathfrak X}$, with complement ${\mathfrak U}$. Then there is a fibration sequence
$$
{\mathbf G}({\mathfrak Y}) \overset{i_*}{\longrightarrow} {\mathbf G}({\mathfrak X}) \overset{j^*}{\longrightarrow} {\mathbf G}({\mathfrak U})
$$
in which $i : {\mathfrak Y} \to {\mathfrak X}$ and $j : {\mathfrak U} \to {\mathfrak X}$ are the inclusions.
\end{lemma}

\begin{proof} Since ${\mathbf G}({\mathfrak X})$ is the hypercohomology ${\mathbf R \Gamma} ({\mathfrak X} , {\mathbf G}_{\mathfrak X})$ of the presheaf ${\mathbf G}_{\mathfrak X} : V \mapsto G(V)$ on the \'etale site of ${\mathfrak X}$, it suffices to observe that for each \'etale morphism $p : V \to {\mathfrak X}$ from a scheme $V$ to ${\mathfrak X}$, we have a localization fibration sequence
$$
{\mathbf G} (V \times_{\mathfrak X} {\mathfrak Y}) \to {\mathbf G}(V) \to {\mathbf G}(V \times_{\mathfrak X} {\mathfrak U})
$$
and hence an isomorphism
$$
{\mathbf R \Gamma}_{\mathfrak Y} ({\mathbf G}) \approx i_* \, {\mathbf G}_{\mathfrak Y} \, .
$$
Note that this also follows easily from the previous lemma and the fact that Holim preserves fibration sequences.
\end{proof}

The functoriality of $\mathbf{G}$ for morphisms of stacks is a more complicated question.
In \cite{Gillet-Intersection-theory-on-stacks} pushforward maps were only constructed for proper
{\it representable} morphisms of stacks.  However we may replace stacks by simplicial varieties, as follows:

Let $p : X. \to \mathfrak{X}$ be a proper morphism to a stack from  a simplicial variety with proper face maps.
Let $\pi : V. \to  \mathfrak{X}$ be an \'etale hypercover of $\mathfrak{X}$,
such as the nerve ${\rm Cosk}^{\mathfrak{X}}_0(V)$ of an \'etale presentation $V \to \mathfrak{X}$. Then we have a commutative square:
$$
\begin{CD}
X. @<\tilde{\pi}<< X.\times_\mathfrak{X} V. \\
@VpVV  @V\tilde{p}VV \\
\mathfrak{X} @<{\pi}<< V.
\end{CD}
$$
For each $i$, since $X.\times_\mathfrak{X} V_i\to V_i$ is proper, we have a map
$$\mathbf{G}(X.\times_\mathfrak{X} V_i)=\mathrm{hocolim}_j(\mathbf{G}(X_j\times_\mathfrak{X} V_i))\to \mathbf{G}(V_i)$$
which is
contravariant with respect to $i$ (since we have constructed $G$ to be strictly compatible with flat base change).
Therefore, we get a diagram:
$$
\begin{CD}
\mathbf{G}(X.) @>\tilde{\pi}^*>> \mathrm{holim}_i \mathbf{G}(X.\times_\mathfrak{X} V_i) \\
@.  @V\tilde{p}_*VV \\
\mathbf{G}(\mathfrak{X}) @>{\pi}^*>> \mathrm{holim}_i(\mathbf{G}(V.))
\end{CD}
$$
in which the bottom horizontal arrow is a weak equivalence.  Hence we get a map (up to homotopy) $p_*:\mathbf{G}(X.)\to \mathbf{G}(\mathfrak{X})$,
which it is straightforward to check (again using the fact that $G$ is strictly compatible with flat base change)
does not depend on the choice of \'{e}tale hypercover $\pi:V.\to\mathfrak{X}$.  Furthermore, if $f.:Y.\to X.$ is a map of simplicial varieties,
using the fact that push forward commutes with flat base-change, we have that
$$ (p\cdot f)_*=p_*\cdot f_*:\mathbf{G}(Y.)\to  \mathbf{G}(\mathfrak{X})$$
in the rational stable homotopy category.

In order to show that this construction gives an extension of $\mathbf{G}$-theory from simplicial varieties to stacks, we need:
\begin{lemma}
Suppose that $p:X.\to {\mathfrak X}$ is a proper hypercover. Then $p_*:\mathbf{G}(X.)\to \mathbf{G}(\mathfrak{X})$ is a weak equivalence.

\end{lemma}
\begin{proof}
Since both the $G$-theory of $X.$ and of $\mathfrak{X}$ are compatible with localization,
we may proceed  by noetherian induction on $\mathfrak{X}$.
By Proposition~\ref{Proposition113}, every Deligne-Mumford stack has a
non-empty dense open set which is a quotient stack.  It therefore suffices to show that if $\mathfrak{X}=[V/\Gamma]$ for $\Gamma$ a finite group, then
$\mathbf{G}(X.)\to \mathbf{G}(X.\times_\mathfrak{X}\mathrm{Cosk}^\mathfrak{X}_0(V))$ is a weak equivalence.
However $\mathrm{Cosk}^\mathfrak{X}_0(V))$ is just the bar construction $(E.\Gamma\times V)/\Gamma $ for the action of $\Gamma$ on $V$, hence
$$X.\times_\mathfrak{X}\mathrm{Cosk}^\mathfrak{X}_0(V)\simeq (E.\Gamma\times (X.\times_\mathfrak{X} V))/\Gamma\; .$$
It is then straightforward to check, since we are taking $G$-theory with rational coefficients, and since
for each $j$,
$X_j\times_\mathfrak{X}V\to X_j$ is a \emph{finite} \'{e}tale Galois cover with group $\Gamma$, that
$$\mathrm{holim}_j(j\mapsto \mathbf{G}(E_j\Gamma\times (X_j\times_\mathfrak{X} V))/\Gamma)$$
is simply the $\Gamma$-invariants of $\mathbf{G}(X.\times_\mathfrak{X}V)$ and the map
$$\mathbf{G}(X.)\to \mathrm{holim}_j(j\mapsto \mathbf{G}(E_j\Gamma\times (X_j\times_\mathfrak{X} V))/\Gamma)$$
is a weak equivalence with inverse induced by the transfer $\mathbf{G}(X.\times_\mathfrak{X}V)\to \mathbf{G}(X.)$ (divided by the order of $\Gamma$).
\end{proof}

\begin{thm}\label{stacks}
Let $f : \mathfrak{X} \to \mathfrak{Y}$ be a proper, not necessarily representable, morphism of stacks.
Then there exists a canonical map (in the homotopy category)
$$\mathbf{G}(\mathfrak{X})\to \mathbf{G}(\mathfrak{Y})$$ with the property that for any commutative square:
$$
\begin{CD}
X. @>\tilde{f}>> Y.\\
@V{p.}VV   @VV{q.}V\\
\mathfrak{X} @>{f}>> \mathfrak{Y}
\end{CD}
$$
in which $p.$ and $q.$ are proper morphisms with domains simplicial schemes with proper face maps,
we have a commutative square in the stable homotopy category:

$$
\begin{CD}
\mathbf{G}(X.) @>\tilde{f_*}>> \mathbf{G}(Y.)\\
@V{\mathbf{G}(p.)}VV   @VV{\mathbf{G}(q.)}V\\
\mathbf{G}(\mathfrak{X}) @>{\mathbf{G}(f)}>> \mathbf{G}(\mathfrak{Y})
\end{CD}
$$

\end{thm}

\begin{proof}
To define the map $f_*:\mathbf{G}(\mathfrak{X}) \to \mathbf{G}(\mathfrak{Y})$, pick any proper hypercover $p:X.\to \mathfrak{X}$, and set
$f_*=(f\cdot p)_*\cdot(p_*)^{-1}$.  Since $p_*$ is functorial with respect to $p:X.\to \mathfrak{X}$, the map $f_*$
does not depend on the choice of $p$.  Given a commutative square as above,
the fact that associated square of $G$-theory spectra is also commutative is a
consequence of the same functoriality for simplicial schemes over $\mathfrak{Y}$.
\end{proof}

\begin{defn}
We shall call a morphism of stacks $f : \mathfrak{X} \to {\mathfrak{Y}}$ a $\mathbf{G}$-equivalence if $\mathbf{G}(f)$ is an equivalence.
\end{defn}

\begin{thm}
\label{stack and coarse space have same k-theory}
If $\mathfrak{X}$ is a stack, the natural map $\mathfrak{X} \to {|\mathfrak{X}|}$  from the stack to its coarse
moduli space is a $\mathbf{G}$-equivalence.
\end{thm}

\begin{proof} Assume $G$ is a finite group acting on some affine scheme $U$, and
$H \subset G$ is a normal subgroup acting trivially on $U$ and such that $G/H$ acts freely on $U$. A standard transfer argument shows that the map $\mathbf{G}([U/G]) \to \mathbf{G}(U/G)$ is an equivalence.
Since every stack has a dense open subset which is of this type one can use localization and
noetherian induction to conclude the proof. \hfil $\Box$

\section{Weight complexes for varieties and stacks}

\subsection{ Introduction }

In this section we prove the extension of the main theorem of \cite{Gillet-Soule-motives-descent} for
varieties and stacks over $S$, where $S$ is a base scheme satisfying the conditions of the introduction.
In particular, this includes the cases $S = {\rm Spec} ({\mathcal O}_K)$  for ${\mathcal O}_K$
the ring of integers in a number field, and $S = {\rm Spec} (k)$ for $k$ a field of characteristic different from zero.

As in section 5 of {\it op. cit.} we shall use $K_0$-motives (but with rational coefficients)
rather than Chow motives. (Manin, in \cite{Manin}, seems to have been the first to mention using
$K_0$-motives.) The proofs in this section are often variations on those in \cite{Gillet-Soule-motives-descent}, but we shall give proofs again where the current situation merits it.

\subsection{$K_0$-correspondences}

\begin{defn}
If $X$ and $Y$ are regular, projective, $S$-varieties, we will write $KC_S (X,Y)$ for $G_0 (X \times_S Y)_{\mathbb Q}$, and call it the group of ``$K_0$-correspondences from $X$ to $Y$''.
\end{defn}

Note that $X \times_S Y$ is not in general regular; however since $Y$ is regular, ${\mathcal O}_Y$ is of finite global tor-dimension, and hence there is a bilinear product, given  $X$, $Y$ and $Z$ regular, projective, $S$-varieties :
$$
* : G_0 (X \times_S Y) \times G_0 (Y \times_S Z) \to G_0 (X \times_S Y \times_S Z)
$$
$$
([{\mathfrak F}], [{\mathfrak G}]) \mapsto \sum_{i \geq 0} (-1)^i [ {\mathcal T}\!\!or_i^{{\mathcal O}_Y} ({\mathfrak F} , {\mathfrak G})] \, .
$$
Composing with the direct image map ($p_{XZ} : X \times_S Y \times_Z Z \to X \times_S Z$ being the natural projection):
$$
(p_{XZ})_* : G_0 (X \times_S Y \times_S Z) \to G_0 (X \times_S Z)
$$
we get a bilinear pairing:
$$
G_0 (X \times_S Y) \times G_0 (Y \times_S Z) \to G_0 (X \times_S Z)
$$
and hence:
$$
KC_S (X,Y) \times KC_S (Y,Z) \to KC_S (X,Z) \, .
$$
The proofs of the following lemmas are straightforward so we omit them.

\begin{lemma}
Given regular varieties $X,Y,Z$ and $W$ projective over $S$, and elements $\alpha \in KC_S (X,Y)$, $\beta \in KC_S (Y,Z)$, $\gamma \in KC_S (Z,W)$ we have
$$
\gamma \circ (\beta \circ \alpha) = (\gamma \circ \beta) \circ \alpha \, .
$$
\end{lemma}

Given a morphism $f : X \to Y$ of regular varieties, projective over $S$, with graph $\Gamma_f \subset X \times_S Y$, we write $\Gamma (f)$ for the class $[{\mathcal O}_{\Gamma (f)}] \in KC_S (X,Y)$.

\begin{lemma}
If $f : X \to Y$ and $g : X \to Y$ are morphisms of regular varieties, projective over $S$, we have
$$
\Gamma (g \circ f) = \Gamma (g) \circ \Gamma (f) \, .
$$
Furthermore, if $X$ and $Y$ are regular projective varieties over $S$, and $\alpha \in KC_S (X,Y)$, then $\Gamma (1_Y) \circ \alpha = \alpha \circ \Gamma (1_X) = \alpha$.
\end{lemma}

\begin{defn}
We write $KC_S$ for the category with objects regular projective varieties over $S$, and homsets the $KC_S (X,Y)$
for $X$ and $Y$ objects in $KC_S$, and identity $\Gamma (1_X) \in KC_S (X,X)$ for each $X$.
\end{defn}

Clearly $\Gamma$ is a covariant functor from the category of regular projective varieties over $S$ to the
category $KC_S$. Note that $KC_S$ is a $\mathbb{Q}$-linear category.

Observe that $\Gamma$ extends to a functor from the category of simplicial varieties which, in each degree,
are regular and projective over $S$ to the category of chain complexes in $KC_S$ by associating (in the usual fashion) to $X.$ the complex
$$
n \mapsto \Gamma (X_n)
$$
with differential $\underset{i=0}{\overset{n}{\sum}} \, (-1)^i(d_i)_* : \Gamma (X_n) \to \Gamma (X_{n-1})$.

\begin{lemma}\label{G-theory-is-motivic}
The functors $G_i$ for $i \geq 0$, where $G_i (X)$ is $K_i$ of the category of coherent sheaves of ${\mathcal O}_X$ modules, factor through $\Gamma$.
\end{lemma}

\begin{proof} Give $\alpha \in KC_S (X,Y)$, we need to define $G_i (\alpha) : G_i (X) \to \, G_i (Y)$. We start by observing that for $X$ regular, $K_i (X) \simeq G_i (X)$, and so if $p_X : X \times_S Y \to X$ and $p_Y : X \times_S Y \to Y$ are the projective, we can define
$$
G_i (\alpha) : x \mapsto p_{Y*} (p_X^* (x) \cap \alpha)
$$
where
$$
\cap : K_i (X \times_S Y) \otimes G_0 (X \times_S Y) \to G_i (X \times_S Y)
$$
is the natural cap product.

Let $Z$ be a regular projective variety and $\beta \in KC_S (Y,Z)$. We want to show that
$$
G_i (\beta) \circ G_i (\alpha) = G_i (\beta \circ \alpha) \, . \leqno ({\rm B}).
$$

Consider the diagram of maps
$$
\xymatrix{
&&&X \times_S Y \times_S Z \ar[dll]_p \ar[d]^q \ar[drr]^r \\
&X \times_S Y \ar[dl]_{p_X} \ar[dr]^{p_Y}  &&Y \times_S Z \ar[dl]^{q_Y} \ar[dr]_{q_Z} &&X \times_S Z \ar[dl]^{r_Z} \ar[dr] ^{r_X}\\
X &&Y &&Z &&X
}
$$
and call $u : X \times_S Y \times_S Z \to X$ the obvious projection. If $x \in K_i (X)$ we have
$$
G_i (\beta) \circ G_i (\alpha) \, (x) = q_{Z*} (q_Y^* (p_{Y*} (p_X^* (x) \cap \alpha)) \cap \beta)
$$
and
\begin{align}
G_i (\beta \circ \alpha) (x) &= &r_{Z*} (r_X^* (x) \cap (r_* (\alpha * \beta))) \nonumber \\
&= &r_{Z*} (r_* (u^* (x) \cap (\alpha * \beta))) \qquad \mbox{(projection formula for $r$)} \nonumber \\
&= &q_{Z*} \, q_* (u^* (x) \cap (\alpha * \beta)) \, . \nonumber
\end{align}
Letting $\xi = p_X^* (x)$, it is enough to show that
$$
q_Y^* (p_{Y*} (\xi \cap \alpha)) \cap \beta = q_* (p^* (\xi) \cap (\alpha * \beta)) \, . \leqno ({\rm A})
$$

Let ${\mathcal F}^{\cdot}$ be a complex of coherent sheaves on $X \times_S Y$ which is acyclic with
respect to $p_Y$ and represents $\alpha \in G_0 (X \times_S Y)$. For any locally free sheaf ${\mathcal E}$ on $X \times_S Y$,
the complex ${\mathcal E} \otimes {\mathcal F}^{\cdot}$ is still acyclic with respect to $p_Y$ and
its derived direct image by $p_Y$ is $p_{Y*} ({\mathcal E} \otimes {\mathcal F}^{\cdot})$. On the other hand, let
$Z \subset {\mathbb P}_S^n$ be a projective embedding of $Z$ and ${\mathcal G}^{\cdot}$ a complex of coherent sheaves on
$Y \times_S {\mathbb P}_S^N$ which is flat on ${\mathcal O}_Y$ and acyclic outside $Y \times_S Z$,
and which represents $\beta$ in the $K$-theory with supports
$$
K_0^{Y \times_S Z} (Y \times_S {\mathbb P}_S^N) = G_0 (Y \times_S Z) \, .
$$
Let $\tilde q_Y : Y \times_S {\mathbb P}_S^N \to Y$,
$\tilde q : X \times_S Y \times_S {\mathbb P}_S^N \to Y \times_S {\mathbb P}_S^N$ and
$\tilde p : X \times_S Y \times_S {\mathbb P}_S^N \to X \times_S Y$ be the obvious projections. By flat base change we know that
$$
\tilde q_Y^* (p_{Y*} ({\mathcal E} \otimes {\mathcal F}^{\cdot})) = \tilde q_* \, \tilde p^* ({\mathcal E} \otimes {\mathcal F}^{\cdot}) \, .
$$
Therefore, by the projection formula for $\tilde q$, and since ${\mathcal G}$ is flat over ${\mathcal O}_Y$,
\begin{align}
\tilde q_Y^* (p_{Y*} ({\mathcal E} \otimes {\mathcal F}^{\cdot})) \underset{{\mathcal O}_{Y \times_S {\mathbb P}_S^N}}{\otimes} {\mathcal G}^{\cdot} &= & \tilde q_* (\tilde p^* ({\mathcal E} \otimes {\mathcal F}^{\cdot}) \underset{{\mathcal O}_{X \times_S Y \times_S {\mathbb P}_S^N}}{\otimes} \tilde q^* ({\mathcal G}^{\cdot})) \nonumber \\
&= &\tilde q_* (\tilde p^* ({\mathcal E}) \otimes \tilde p^* ({\mathcal F}^{\cdot}) \otimes \tilde q^* ({\mathcal G}^{\cdot})) \, . \nonumber
\end{align}
The functor ${\mathcal E} \mapsto \tilde q_Y^* (p_{Y*} ({\mathcal E} \otimes {\mathcal F}^{\cdot})) \otimes {\mathcal G}^{\cdot}$ induces the map $\xi \mapsto q_Y^* (p_{Y*} (\xi \cap \alpha)) \cap \beta$ on higher $K$-theory, and the functor
$$
{\mathcal E} \mapsto \tilde q_* (\tilde p^* ({\mathcal E}) \otimes \tilde p^* ({\mathcal F}^{\cdot}) \otimes \tilde q^* ({\mathcal G}^{\cdot}))
$$
induces the map $\xi \mapsto q_* (p^* (\xi) \cap (\alpha * \beta))$. Therefore (A) and (B) follow. \end{proof}

\begin{defn}
We set $\mathbf{KM}_S$ equal to the idempotent completion of $KC_S$. Note that these are {\rm homological} motives.
\end{defn}

The proof of the following theorem, given Lemma~\ref{G-theory-is-motivic}, is the same as that of Theorem~6 of \cite{Gillet-Soule-motives-descent} (which itself is a variation of Theorem~1 of {\it op. cit.}):

\begin{thm}
\label{universal}
Let
$$
\xymatrix{
X. \ar[d]_p \ar[r]^g &Y. \ar[d]^q  \\
Z. \ar[r]_f  &W. \\
}
$$
be a commutative square of maps between simplicial objects in the category of regular projective varieties over $S$. Suppose that for all regular projective varieties $V$ over $S$, the associated square of spectra:
$$
\xymatrix{
{\mathbf G} (V \times_S X.) \ar[d] \ar[r] &{\mathbf G} (V \times_S Y.) \ar[d]  \\
{\mathbf G} (V \times_S Z.) \ar[r] &{\mathbf G} (V \times_S W.) \\
}
$$
is homotopy cartesian. Then the associated square of complexes in $KC_S$ is homotopy cartesian, {\rm i.e.}, the associated total complex is contractible, or equivalently the natural map ${\rm Cone} (\Gamma_* (g.)) \to {\rm Cone} (\Gamma_* (f.))$ is a homotopy equivalence.
\end{thm}

\subsection{Weight complexes for simplicial varieties}

Following section~2.2 of \cite{Gillet-Soule-motives-descent}, we write ${\rm Ar} (P_S)$ for the category of morphisms in $P_S$, where $P_S$ is the category of proper, not necessarily regular $S$-varieties, with objects $f : Y \to X$, and morphisms $g : f' \to f$ commutative squares :
$$
\xymatrix{
Y'\ar[d]_{g_Y} \ar[r]^{f'} &X' \ar[d]^{g_X}  \\
Y \ar[r]^f  &X \\
}
$$
We can also consider the category ${\rm Ar} (P_S^{\Delta^{\rm op}})$ of arrows between simplicial objects on $P_S$. A morphism
$$
g. : (Y'_{\cdot} \overset{f'_{\cdot}}{\longrightarrow} X') \to (Y. \overset{f.}{\longrightarrow} X.)
$$
in ${\rm Ar} (P_S^{\Delta^{\rm op}})$ will be called a proper hypercover if both $g_{Y.}$ and $g_{X.}$ are proper hypercovers. We have a functor
$$
\Gamma_* : {\rm Ar} (RP_S^{\Delta^{\rm op}}) \to C_* (KC_S)
$$
from the category of arrows between simplicial {\it regular} projective varieties to the category of chain complexes in $KC_S$,
$$
\Gamma_* : (f. : X. \to Y.) \mapsto {\rm Cone} (\Gamma_* (f.)) \, .
$$
We shall now show, following {\it op. cit.}, that this functor induces a functor from ${\rm Ar} (P_S^{\Delta^{\rm op}})$ to the homotopy category of $C_* (KC_S)$.

\begin{thm}
\label{maps}
The functor
\begin{align}
{\rm Ar} (RP_S^{\Delta^{\rm op}}) &\to &H_0 (C_* (KC_S)) \nonumber \\
f. &\mapsto &\Gamma_* (f.) \nonumber
\end{align}
has a unique extension to ${\rm Ar} (P_S^{\Delta^{\rm op}})$ with the property that proper hypercovers map to homotopy equivalences.
\end{thm}

\begin{proof} The proof is the same as in \cite{Gillet-Soule-motives-descent}, \S~2.2, using Theorem~\ref{proper} instead of Hironaka's resolution of singularities.
 The theorem is also Theorem 5.3.c) of \cite{KS}, applied to $D^{\rm op}$ the
category $P_S^{\Delta^{\rm op}}$, to $C^{\rm op}$ the full subcategory $RP_S^{\Delta^{\rm op}}$, to $S$ the category of proper hypercovers,
 and to $E = \Delta^1$. (Note that Theorem~5.3 in \cite{KS} was partly inspired by \cite{Gillet-Soule-motives-descent}.) \end{proof}

We shall call a morphism $g.$ in ${\rm Ar} (P_S^{\Delta^{\rm op}})$ a {\it universal $G$-equivalence} if after multiplying all schemes involved by any regular projective variety $V$, the corresponding square of spectra obtained by applying $G_{\mathbb Q}$ is homotopy cartesian. Theorem \ref{universal} tells us that if $g.$ is a universal $G$-equivalence and all schemes involved are regular then $\Gamma_* (g.)$ is a homotopy equivalence.

Using Theorem~1 and the same proof as in \cite{Gillet-Soule-motives-descent}, end of \S~2.2, we get from this that if $g.$ is a universal $G$-equivalence among arbitrary projective schemes, then $\Gamma_* (g.)$ is a homotopy equivalence.

\begin{thm}
\label{simplicial}
There is a covariant functor
$$
h : {\mathcal V}ar_S^{\Delta^{\rm op}} \to {\rm Ho} (C_* (KC_S))
$$
from the category of simplicial varieties over $S$ with proper face maps and proper morphisms to the category of homotopy classes of maps of complexes of $K_0$-motives over $S$ with rational coefficients, satisfying the following properties:
\begin{itemize}
\item[{\rm i)}] If, for all $n \geq 0$, $X_n$ is a regular projective scheme over $S$ and $h(X_n)$ is the usual motive of $X_n$, $h(X.)$ is the complex of motives
$$
\xymatrix{
\ldots \to h(X_n) \quad \ar[r]^{\underset{i=0}{\overset{n}{\sum}} \, (-1)^i \, d_{i*}} &\quad h(X_{n-1}) \to \ldots
}
$$
\item[{\rm ii)}] If $U. \subset X.$ is a strongly open simplicial subvariety of $X.$ with proper face maps and with complement $T.$, then we have a triangle
$$
h(T.) \to h(X.) \to h(U.) \to h(T.) [1] \, .
$$
\item[{\rm iii)}] If $X. \to Y.$ is a proper hypercover, the induced map $h(X.) \to h(Y.)$ is a homotopy equivalence.
\end{itemize}
\end{thm}

\begin{proof} Let $X.$ be a simplicial variety over $S$ with proper face maps. According to Proposition~\ref{split} there exists a proper hypercover $\tilde X. \to X.$ with $\tilde X.$ a split simplicial variety. By Proposition~\ref{compact} $\tilde X.$ admits a compactification $\bar X.$ with complement $\bar X. \backslash \tilde X.$. Let $h(X.)$ be the image in ${\rm Ho} (C_* (KC_S))$ of the arrow $\bar X. \backslash \tilde X. \to \bar X.$ (Theorem~\ref{maps}). We claim that $h(X.)$ is functorial in $X.$ and does not depend on the choices made to define it.

Indeed, let $f : X. \to Y.$ be a map in ${\mathcal V}ar_S^{\Delta^{\rm op}}$ and $\tilde X. \to X.$ and $\tilde Y. \to Y.$ proper hypercovers admitting compactifications $\bar X.$ and $\bar Y.$. The fiber product $\tilde X. \times_{Y.} \tilde Y.$ has a compactification $\overline{X. \times Y.}$ which is the Zariski closure of $\tilde X. \times_{Y.} \tilde Y.$ in $\bar X. \times_S \bar Y.$. Then we get two commutative squares
$$
\xymatrix{
\overline{X. \times Y.} \backslash \tilde X. \times \tilde Y. \ar[d] \ar[r] &\overline{X. \times Y.} \ar[d] \\
\bar X. \backslash \tilde X. \ar[r] &\bar X. \\
}
$$
and
$$
\xymatrix{
\overline{X. \times Y.} \backslash \tilde X. \times \tilde Y. \ar[d] \ar[r] &\overline{X. \times Y.} \ar[d] \\
\bar Y. \backslash \tilde Y. \ar[r] &\bar Y. \, , \\
}
$$
the first of which is a universal $G$-equivalence. Hence we get a map
$$
h(f) : h(X.) = \Gamma_* (\bar X. \backslash \tilde X. \to \bar X.) \to \Gamma_* (\bar Y. \backslash \tilde Y. \to \bar Y.) = h(Y.) \, .
$$
When $X. = Y.$ and $f$ is the identity, both squares above are universal $G$-equivalences and it follows that $h(X.)$ does not depend on choices made to define it.

Assume now that $f : X. \to Y.$ and $g : Y. \to Z.$ are proper maps of simplicial varieties. Then the same argument as in \cite{Gillet-Soule-motives-descent}, \S2.3 (except that we replace ``Gersten acyclic'' by ``universal $G$-equivalence'') shows that $h(g) \circ h(f) = h (g \circ f)$, {\it i.e.} $h(X.)$ is functorial in $X.$.

When every $X_n$ is regular and projective over $S$, $n \geq 0$, we can take $\tilde X. = \bar X. = X.$, so Property i) is clear.

To check ii), let $\tilde X.$ be a proper hypercover of $X.$, with compactification $\bar X.$. Let $\tilde U.$ (resp. $\tilde T.$) be the inverse image of $U.$ (resp. $T.$) in $\tilde X.$ and define $Y. = \bar X. - \tilde X.$, $Z. = \bar X. - \tilde U.$. We have inclusions $Y. \to Z. \to \bar X.$. Consider a proper hypercover of this diagram
$$
\xymatrix{
Y'. \ar[d] \ar[r] &Z'. \ar[d] \ar[r] &X'. \ar[d] \\
Y. \ar[r] &\bar Z. \ar[r] &\bar X. \\
}
$$
with $Y'_{\cdot}$, $Z'_{\cdot}$ and $X'_{\cdot}$ in $RP_S$. We get a triangle
$$
\Gamma_* (Y'_{\cdot} \to Z'_{\cdot}) \to \Gamma_* (Y'_{\cdot} \to X'_{\cdot}) \to \Gamma_* (Z'_{\cdot} \to X'_{\cdot}) \to \Gamma_* (Y'_{\cdot} \to Z'_{\cdot}) [+1] \, .
$$
Since $Z. \backslash Y. = \tilde T.$ and $\bar X.$ is a compactification of both $\tilde X.$ and $\tilde U.$, using Theorem~\ref{maps}, we can write this triangle as
$$
h(T.) \to h(X.) \to h(U.) \to h(T.) [1] \, .
$$

Finally, to check iii), if $f : X. \to Y.$ is a proper hypercover, $\tilde X. \to X.$ a proper hypercover and $\bar X.$ a compactification of $\tilde X.$, we notice that the composite map $\tilde X. \to Y.$ is a proper hypercover so that $h(Y.) = \Gamma_* (\bar X. \backslash \tilde X. \to \bar X.) = h(X.)$.

\subsection{Weight complexes of stacks}

We now prove Theorem \ref{motstacks} of the introduction. Given a
stack ${\mathfrak X}$ and $X. \to {\mathfrak X}$ a proper hypercover
by a simplicial variety we define $h({\mathfrak X}) = h(X.)$.

Let $f : {\mathfrak X} \to {\mathfrak Y}$ be a proper map of stacks and $X. \to {\mathfrak X}$ and $Y. \to {\mathfrak Y}$ proper hypercovers of these. According to Lemma~\ref{base change hypercovers stacks} the induced map $X. \times_{\mathfrak Y} \, Y. \to X.$ is a proper hypercover. Therefore, by Theorem~\ref{simplicial} iii), we get a map
$$
h(f) : h({\mathfrak X}) = h(X.) = h (X. \times_{\mathfrak Y} \, Y.) \to h(Y.) = h({\mathfrak Y}) \, .
$$
If $g : {\mathfrak Y} \to {\mathcal Z}$ is a proper map of stacks one easily checks that $h(g \circ f) = h(g) \circ h(f)$. When $f = {\rm id}_{\mathfrak X}$ we get $h(f) = {\rm id}_{h({\mathfrak X})}$, {\it i.e.} $h({\mathfrak X})$ does not depend on the choice of $X.$.

\begin{remark} If ${\mathfrak X}$ is a stack, then for any variety (and hence for any simplicial variety) $X$, ${\rm Hom}_S (X,{\mathfrak X})$ is a groupoid. If $f : X'_{\cdot} \to {\mathfrak X}$ and $g : X''_{\cdot} \to {\mathfrak X}$ are two proper hypercovers, then an isomorphism $\theta : f \to g$ induces a section $\tilde\theta : X'_{\cdot} \to X'_{\cdot} \times_{\mathfrak X} \, X''_{\cdot}$ of the projection $p' : X'_{\cdot} \times_{\mathfrak X} \, X'' \to X'_{\cdot}$ and hence the map $h(X'_{\cdot}) \to h(X''_{\cdot})$ induced by $\theta$ coincides with the homotopy equivalence $h(p''_{\cdot}) \, h(p'_{\cdot})^{-1}$, where $p''_{\cdot} : X'_{\cdot} \times_{\mathfrak X} \, X''_{\cdot} \to X''_{\cdot}$ is the second projection. Therefore the functor ${\mathfrak X} \to h({\mathfrak X})$ from the 2-category of stacks to the category of homotopy classes of maps between complexes of motives maps each 2-morphism to the identity.
\end{remark}

If ${\mathfrak Y} \subset {\mathfrak X}$ is a closed substack with complement ${\mathfrak U}$ and if $X. \overset{f}{\longrightarrow} {\mathfrak X}$ is a proper hypercover, $U. = f^{-1} ({\mathfrak U})$ is strongly open in $X.$ with complement a proper hypercover of ${\mathfrak Y}$. Therefore Theorem~\ref{simplicial} ii) gives a triangle
$$
h({\mathfrak Y}) \to h({\mathfrak X}) \to h ({\mathfrak U}) \to h ({\mathfrak Y}) [1] \, .
$$

If $G$ is a finite group acting on a regular projective scheme $X$, and $[X/G]$ the quotient stack, a proper hypercover of $[X/G]$ is the simplicial scheme ${\rm Cosk}_0^{[X/G]}$ $(X)$, which is isomorphic to $(X \times EG)/G$, where $EG$ is the standard contractible simplicial set with free action of $G$. In degree $k \geq 1$ we have
$$
h(X \times G^k) = \underset{(g_1 , \ldots , g_k) \in G^k}{\oplus} \, h(X)
$$
and so $h((X \times EG) / G)$ is the chain complex computing the homology groups $H_* (G,h(X))$ in the Karoubian category of motives. Since the order of $G$ is invertible, this homology vanishes except for $H_0 (G,h(X)) = h(X)^G$, the image of the projector $\frac{1}{\# \, G} \ \underset{g \in G}{\sum} \ g_*$.

To check that $h({\mathfrak X})$ is homotopy equivalent to a bounded complex we first need a lemma.

\begin{lemma}
If $f : {\mathfrak X} \to {\mathfrak Y}$ is a finite, representable radicial map of stacks, the induced map $\mathbf{G}({\mathfrak X}) \to \mathbf{G}({\mathfrak Y})$ is a weak equivalence.
\end{lemma}

\noindent {\it Proof of Lemma.} Let $Y. \to {\mathfrak Y}$ be a proper hypercover of ${\mathfrak Y}$. Since $f$ is representable, $X. : {\mathfrak X} \times_{\mathfrak Y} \, Y.$ is a simplicial scheme and the natural map $X. \to {\mathfrak X}$ is a proper hypercover. Let $g : X. \to Y.$ be the induced map.
Since $f$ is finite and radicial each $g_i : X_i \to Y_i$ is finite and radicial and so
(\cite{QuillenHAKTI}) $g_{i*} : \mathbf{G}(X_i) \to \mathbf{G}(Y_i)$ is a weak equivalence.
Taking homotopy colimits we get that $g_* : \mathbf{G}(X.) \to \mathbf{G}(Y.)$ is a weak equivalence,
and by \ref{stacks} $\mathbf{G}({\mathfrak X}) \to \mathbf{G}({\mathfrak Y})$ is a weak equivalence. \end{proof}

To prove that $h({\mathfrak X})$ is bounded we can assume, by noetherian induction, that ${\mathfrak X}$ is irreducible. By Proposition~\ref{Proposition113}, there is a dense open ${\mathfrak U} \subset {\mathfrak X}$ which is a quotient stack and, by noetherian induction, we just need to show that $h({\mathfrak U})$ is bounded. The quotient stack ${\mathfrak U}$ admits some equivariant compactification (Lemma~\ref{compactify quotient stacks}) $[V/G] = {\mathcal V}$ proper over $S$. By induction on dimension it is enough to show that $R({\mathcal V})$ is bounded. By resolution of singularities (Theorem~\ref{quotient}) there exists a proper morphism of quotient stacks
$$
p : [Y/H] \to {\mathcal V}
$$
where $Y$ is regular and integral and a dense open substack ${\mathcal W} \subset {\mathcal V}$ such that the induced map $p^{-1} ({\mathcal W}) \to {\mathcal W}$ is representable and radicial. We need to show that $h({\mathcal V})$ is bounded {\it i.e.}, again by induction, that $h({\mathcal W})$ is bounded. From the previous lemma, $p^{-1} ({\mathcal W}) \to {\mathcal W}$ is a universal $G$-equivalence, therefore $h({\mathcal W}) = h (p^{-1} ({\mathcal W}))$. But, since $Y$ is proper and regular $h([Y/H])$ is bounded hence, by induction on dimensions, $h(p^{-1} ({\mathcal W}))$ is bounded.

This ends the proof of Theorem \ref{motstacks}.
\end{proof}

Corollary \ref{Euler} follows from Theorem \ref{motstacks} be letting $\chi_c(\mathfrak{X})$
be the class in $K_0(\mathbf{KM}_S)$ of the complex $h(\mathfrak{X})$ (compare \cite{Gillet-Soule-motives-descent}, Lemma 3 and Theorem 4).

\subsection{Chow motives}
If $S=\mathrm{Spec}(k)$ with $k$ a perfect field (of arbitrary characteristic), there are three possible categories of motives that we could consider:
\begin{enumerate}
 \item The category $\mathbf{KM}_k$ of $K_0$-motives over $\mathrm{Spec}(k)$.
 \item The category $\mathbf{CHM}^0_k$ of effective Chow motives over $\mathrm{Spec}(k)$, which is the pseudo-abelian completion of the additive category in which morphisms are \emph{degree zero} correspondences, modulo linear equivalence, between smooth projective varieties over $k$, as in section 5 of \cite{Manin}.
  \item The category $\mathbf{CHM}_k$ which is the pseudo-abelian completion of the additive category in which morphisms are  \emph{all} correspondences, modulo linear equivalence,  between smooth projective varieties over $k$.
\end{enumerate}
We can also take any or all of these categories with rational coefficients, getting categories
$\mathbf{KM}_{k\mathbb{Q}}$, $\mathbf{CHM}^0_{k\mathbb{Q}}$ and $\mathbf{CHM}_{k\mathbb{Q}}$.
\begin{remark}
If we do not assume that $k$ is perfect then we should take the category of all regular projective varieties over $k$. We can still define $\mathbf{CHM}^0_k$ and $\mathbf{CHM}_k$ as above.  Since the product of two regular varieties will no longer necessarily be regular the composition of correspondences is harder to define.
\end{remark}

For any smooth variety $V$ over $k$ the Chern character
$ch:K_0(V)\to \mathrm{CH}^*(V)_\mathbb{Q}$ induces an isomorphism
$ch:K_0(V)_\mathbb{Q}\to \mathrm{CH}^*(V)_\mathbb{Q}$. 
Therefore, given varieties $X$ and $Y$,  the map:
\begin{eqnarray*}
\tau:K_0(X\times Y)_\mathbb{Q} &\to& \mathrm{CH}^*(X\times Y)_\mathbb{Q}\\
     \alpha &\mapsto & ch(\alpha)p_Y^*(\mathrm{Td}(Y))
\end{eqnarray*}
where $p_Y:X\times Y\to Y$ is the projection onto the second factor, is an isomorphism of $\mathbb{Q}$-vector spaces, since the Todd genus is a unit in the Chow ring. 

Suppose that $\alpha\in K_0(X\times Y)$ and $\beta\in K_0(Y\times Z)$ are $K_0$-correspondences from $X$ to $Y$ and from $Y$ to $Z$ respectively. Then, by the Grothendieck Riemann-Roch theorem, \cite{Borel-Serre}, we have:
\begin{eqnarray*}
\tau(\beta\cdot\alpha) &=& ch(p_{XZ*}(p_{XY}^*(\alpha)p_{YZ}^*(\beta)))p_Z^*(\mathrm{Td}(Z))\\
 &=&  p_{XZ*}(ch(p_{XY}^*(\alpha)p_{YZ}^*(\beta))p_Y^*(\mathrm{Td}(Y))p_Z^*(\mathrm{Td}(Z)))\\
&=&  p_{XZ*}(p_{XY}^*ch(\alpha)p_{YZ}^*ch(\beta)p_Y^*(\mathrm{Td}(Y))p_Z^*(\mathrm{Td}(Z)))\\
&=& p_{XZ*}(p_{XY}^*\tau(\alpha)p_{YZ}^*\tau(\beta))\\
&=& \tau(\beta)\cdot\tau(\alpha).
\end{eqnarray*}
Therefore $\tau$ is compatible with composition of correspondences, and hence gives an isomorphism of categories:
$$\mathbf{KM}_{k\mathbb{Q}}\to \mathbf{CHM}_{k\mathbb{Q}}.$$
If $f:X\to Y$ is a morphism, and $\gamma_f:X\to X\times Y$ is the induced map, which is an isomorphism onto the graph $\Gamma(f)$ of $f$, since the normal bundle of $\Gamma(f)\subset X\times Y$ is isomorphic to $f^*(T_Y)$ (the pull back of the tangent bundle of $Y$), the Grothendieck Riemann-Roch theorem for
$\gamma_f$ gives $\tau([\mathcal{O}_{\Gamma(f)}]=[\Gamma(f)]$, and so the functor $\tau$ respects the natural functors from the category of smooth projective varieties to the two categories of motives.

Since the functor $\tau$ is an isomorphism, we have have an isomorphism: 
$$K_0(\mathbf{KM}_{k\mathbb{Q}})\to K_0(\mathbf{CHM}_{k\mathbb{Q}}).$$

Since the degree zero correspondences from $X$ to $Y$ are a subgroup of the group of all correspondences from $X$ to $Y$, and the inclusion is compatible with composition, we get a functor  $\mathbf{CHM}^0_k\to \mathbf{CHM}_k$.  This functor preserves idempotents and is therefore exact, and so induces a homomorphism on the associated Grothendieck groups:
$$K_0(\mathbf{CHM}^0_k)\to  K_0(\mathbf{CHM}_k).$$

From Corollary \ref{Euler} and the isomorphism $K_0(\mathbf{KM}_{k\mathbb{Q}})\to K_0(\mathbf{CHM}_{k\mathbb{Q}})$, we get:
\begin{cor}
Given any reduced variety $X$ over $k$ one can define
an element $\chi_c(X) \in K_0(\mathbf{CHM}_{k\mathbb{Q}})$ in such a way that

\begin{enumerate}
 \item[i)] If $X$ is smooth and projective, $\chi_c(X)$ is the class
of $(X,\Delta_X)$, where $\Delta_X$ is the diagonal
in $X\times X$.
 \item[ii)] If $Y\subset X$ is a closed subset,
the  equality
$$\chi_c(X) = \chi_c(Y) + \chi_c(X\setminus Y)$$
holds in $K_0(\mathbf{CHM}_{k\mathbb{Q}})$.
\end{enumerate}
\end{cor}

One can also prove the existence 
in ${K_0(\mathbf{CHM}^0_{k\mathbb{Q}})}$  
of an Euler characteristic satisfying i) and ii) above (and mapping to $\chi_c(X)$ in $K_0(\mathbf{CHM}_{k\mathbb{Q}})$) by using the method of \cite{Gillet-Soule-motives-descent}. Indeed Theorem \ref{descent} above shows that if $f.:X.\to Y.$ is a proper hypercover between simplicial varieties having proper face maps, the induced map of Gersten complexes tensored with $\mathbb{Q}$ is a quasi-isomorphism.  If we use this fact instead of Proposition 1 in \cite{Gillet-Soule-motives-descent}, the whole argument in \emph{op. cit.} remains valid (after tensoring with $\mathbb{Q}$) and the existence of an Euler characteristic in $K_0(\mathbf{CHM}^0_{k\mathbb{Q}})$ follows as in \cite{Gillet-Soule-motives-descent}, Theorem 4.  When $k$ has characteristic zero, this Euler characteristic is the image of the Euler characteristic in $K_0(\mathbf{CHM}^0{k})$ constructed in \cite{Gillet-Soule-motives-descent},  \cite{Guillen}, \cite{Bittner}.

\subsection{A variant}

Let $\ell$ be a prime integer.

\begin{defn}
If $f : X \to Y$ is proper and surjective morphism, we say that $f$ is an
$\ell'$-{\rm envelope} if, for every point $y \in Y$, there
is an $x \in X$ {\rm s.t.} $f(x) = y$ and $k(y) \subset k(x)$ is a finite algebraic extension of degree prime to $\ell$.
\end{defn}

It follows from a result of Gabber \cite{Gabber} that if $X$ is an irreducible reduced separated scheme flat and of
finite type over ${\rm Spec} (R)$,
for $R$ an excellent
Dedekind ring
with $\ell$ invertible in $R$,
there exists an $\ell'$-envelope $X' \to X$ with $X'$ regular.

If ${\mathcal P}$ is the class of $\ell'$-envelopes; then ${\mathcal P}$ satisfies the conditions of
\S{2.2}, and so we can talk of a map $f. : X. \to Y.$ of simplicial
schemes being a ``hyper $\ell'$-envelope''. One may then show, by the methods of section \ref{descent},
that any homology theory $E$ on the category of varieties taking values in the category of spectra,
localized away from $\ell$, satisfies descent with respect to hyper $\ell'$-envelopes, and that one has a functor, for $S = {\rm Spec} (R)$
$$
{\rm Stack}_S \to {\rm Ho} (C_* (KC_S ; {\mathbb Z}_{(\ell)}))
$$
from the category of Deligne-Mumford stacks of finite type over $S$ to the category of homotopy classes of
complexes of $K_0$-motives over $S$ with ${\mathbb Z}_{(\ell)}$-coefficients with the properties of
Theorem~\ref{motstacks}, {\it except} that we do not know if it takes values in the subcategory of
complexes homotopy equivalent to bounded complexes.

\section{Contravariance of weight complexes}
\subsection{A Category of complexes of sheaves}

In this section ${\mathcal V}$ will be the category of varieties
($=$ reduced separated schemes of finite type) over a field $k$.

\begin{defn}
If $X$ is a variety, we say that a quasi-coherent sheaf $\mathcal F$
on $X$ is \emph{strongly acyclic} if for every morphism of varieties
$f:X\to Y$ and every quasi-coherent sheaf $\mathcal G$ of
$\mathcal{O}_X$-modules,
$$R^if_*(\mathcal{F}\otimes_{\mathcal{O}_X} \mathcal{G})=0$$
for $i>0$. We say that $\mathcal F$ is \emph{universally} strongly
acyclic, if for every morphism of varieties $f:X\to Y$, all
morphisms $Z\to Y$, and every quasi-coherent sheaf $\mathcal G$ of
$\mathcal{O}_{Z\times_Y X}$-modules,
$$R^if_{Z*}(\mathcal{F}\otimes_{\mathcal{O}_{X}} \mathcal{G})=0$$
where $f_{Z}$ is the base change of $f$ by the morphism $Z\to Y$.
\end{defn}
The example that we have in mind is:
\begin{lemma}\label{affines_are_usacyclic}
If $j:U=\mathrm{Spec}(A)\hookrightarrow X$ is an affine open subset,
and if $\mathcal F=j_*\widetilde{M}$, for $\widetilde{M}$ the
quasi-coherent sheaf of ${\mathcal O}_U$-modules associated to the
$A$-module $M$, then $\mathcal F$ is universally strongly acyclic.
\end{lemma}
\begin{proof}
First of all observe that since varieties are separated, any
morphism $g:U\to Y$ of varieties with affine domain is an affine
morphism, and so $R^ig_*(\widetilde{M})=0$ for $i>0$.  It then follows by a standard spectral
sequence argument that $j_*\widetilde{M}$ is acyclic with respect to
any morphism $f:X\to Y$.

Furthermore, if $\mathcal G$ is a quasi-coherent sheaf on $X$, then
$j^*{\mathcal G}$ is quasi-coherent ($=\widetilde{\mathcal{G}(U)}$),
and
$$j_*\widetilde{M}\otimes_{\mathcal{O}_X} {\mathcal G}\simeq j_*(\widetilde{M}\otimes_{\mathcal{O}_U} j^*{\mathcal G}).$$
We see, therefore, that $\widetilde{M}$ is strongly acyclic.
Finally, let $f:X\to Y$ and $Z\to Y$ be morphisms of varieties;
since higher direct images with respect to
$$f_Z:Z\times_Y X\to Z$$
can be computed locally on $Z$, we may suppose that $Z$ is affine.
But then (separation again!) $Z\times_Y U$ is an affine open subset
of  $Z\times_Y X$ and writing $j_Z:Z\times_Y U\to Z\times_Y X$ for
the inclusion, we have:
$$j_*\widetilde{M}\otimes_{\mathcal{O}_X}\mathcal{O}_{Z\times_Y X}\simeq (j_Z)_*(\widetilde{M}\otimes_{\mathcal{O}_U}\mathcal{O}_{Z\times_Y U}) .$$
But $\widetilde{M}\otimes_{\mathcal{O}_U}\mathcal{O}_{Z\times_Y U}$
is again a quasi-coherent sheaf on an affine open, and so
$j_*(\widetilde{M})\otimes_{\mathcal{O}_X}\mathcal{O}_{Z\times_Y X}$
is strongly acyclic.  Thus $j_*\widetilde{M}$ is universally
strongly acyclic.

\end{proof}
Note also, that by a standard spectral sequence argument:
\begin{lemma}
If $f:X\to Y$ is a morphism, and $\mathcal F$ is a strongly acyclic
(resp. universally strongly acyclic) quasi-coherent sheaf of
${\mathcal O}_X$-modules, then $f_*{\mathcal F}$ is strongly acyclic
(resp. universally strongly acyclic).
\end{lemma}
Given varieties $X$ and $Y$, we let ${\mathbf C} (X,Y)$ be the following
Waldhausen category (also referred to as a category with
cofibrations and weak equivalences, but here we follow the
terminology of \cite{Thomason_Trobaugh}).

The objects of ${\mathbf C} (X,Y)$ are bounded complexes of  ${\mathfrak
F}^*$ of quasi-coherent ${\mathcal O}_{X \times Y}$-modules, such that
\begin{itemize}
    \item  each ${\mathfrak F}^i$ is flat over $X$;
    \item the cohomology sheaves of ${\mathfrak F}^{\cdot}$ are coherent with
support proper over $Y$
    \item each ${\mathfrak F}^i$ is universally strongly acyclic.
\end{itemize}
The cofibrations in ${\mathbf C} (X,Y)$ are the monomorphisms
${\mathfrak F}^{\cdot} \to {\mathcal G}^{\cdot}$ with cokernel in
${\mathbf C} (X,Y)$, and the weak equivalencies are the
quasi-isomorphisms.

\begin{lemma}
\label{contra} { \ }

\begin{itemize}
\item[{\rm 1)}] Let $f : X' \to X$ be a morphism of varieties and
${\mathcal F}^{\cdot}$ an object in ${\mathbf C} (X,Y)$. Then $(f \times
1_Y)^* ({\mathcal F}^{\cdot})$ lies in ${\mathbf C} (X' , Y)$.
\item[{\rm 2)}] Let $p : Y \to Y'$ be a proper morphism of varieties.
Then $(1_X \times p)_* ({\mathcal F}^{\cdot})$ lies in ${\mathbf C}
(X,Y')$.
\end{itemize}
\end{lemma}
\begin{proof}
The first assertion is a basic property of sheaves of flat modules.
Turning to the second assertion, given the previous lemma, we need
only show that $(1_X \times p)_* ({\mathcal F}^{\cdot})$ is
$\mathcal{O}_X$-flat.  If
$0\to\mathcal{A}\to\mathcal{B}\to\mathcal{C}\to 0$ is an exact
sequence of $\mathcal{O}_X$-modules, then $0\to\mathcal{A}
\otimes_{{\mathcal O}_X} {\mathcal F}^{\cdot}\to\mathcal{B}
\otimes_{{\mathcal O}_X}{\mathcal F}^{\cdot}\to\mathcal{C}
\otimes_{{\mathcal O}_X}{\mathcal F}^{\cdot}\to 0$ is exact. By
universal acyclicity, the push forward of this sequence is still
exact, and hence $(1_X \times p)_* ({\mathcal F}^{\cdot})$ is still
${\mathcal O}_X$-flat.
\end{proof}

By the hypothesis on the objects in ${\mathbf C} (X,Y)$, both $p_*$ and
$f^*$ are exact functors preserving weak equivalences. Furthermore,
we can rigidify these categories by choosing inverse images with
respect to all $f:X'\to X$. Direct images are already ``rigid'', and
so we get a functor
\begin{align}
{\mathcal V}^{\rm op} \times {\mathcal V} &\to &\mbox{Waldhausen Categories} \nonumber \\
(X,Y) &\mapsto &{\mathbf C} (X,Y) \nonumber
\end{align}

If $X$ is a variety, let us write $\mathbf{C}(X)$ for the Waldhausen
category of complexes of quasi-coherent sheaves of ${\mathcal
O}_X$-modules with bounded coherent cohomology, and weak
equivalences given by quasi-isomorphisms. Recall
\cite{Thomason_Trobaugh} that the $K$-theory of this category is
naturally isomorphic to the $G$-theory of $X$.

If $X$ and $Y$ are varieties, and $p:X\times Y\to X$ and $q:X\times
Y\to Y$ are the projections, observe that if  ${\mathcal F}^\cdot$
is an object of ${\mathbf C} (X,Y)$, then
$${\mathcal G}^\cdot\mapsto q_*(p^*({\mathcal G}^\cdot)\otimes_{\mathcal{O}_{X\times Y}}{\mathcal F}^\cdot) $$
defines an exact functor
$$\mathbf{C}(X)\to \mathbf{C}(Y)$$
and furthermore it is straightforward to check:
\begin{lemma}
$$\mathbf{C}(X)\times {\mathbf C} (X,Y)\to \mathbf{C}(Y)$$
is a bi-exact functor.
\end{lemma}
Equivalently, we have an exact functor from $\mathbf{C} (X,Y)$ to
the exact category $\mathbf{Exact}(\mathbf{C}(X),\mathbf{C}(Y))$ of
exact functors from $\mathbf{C}(X)$ to $\mathbf{C}(Y)$. Not
surprisingly, we have a composition law on the categories
$\mathbf{C} (X,Y)$, which it is easy to see is compatible with
composition:
$$\mathbf{Exact}(\mathbf{C}(X),\mathbf{C}(Y))\times \mathbf{Exact}(\mathbf{C}(Y),\mathbf{C}(Z))\to \mathbf{Exact}(\mathbf{C}(X),\mathbf{C}(Z))\;.$$

\begin{lemma}
Given a triple of proper varieties $X$, $Y$, $Z$, then we have a
bi-exact functor:
\begin{align}
\gamma_{X,Y,Z}:{\mathbf C} (X,Y) \times {\mathbf C} (Y,Z)&\to & {\mathbf C} (X,Z) \nonumber \\
({\mathcal G}^\cdot,{\mathcal F}^\cdot) &\mapsto & r_*(p^*({\mathcal
G}^\cdot)\otimes_{\mathcal{O} _Y}q^*({\mathcal
F}^\cdot))\nonumber\;,
\end{align}
where $p:X\times Y\times Z\to X\times Y$ , $q:X\times Y\times Z\to
Y\times Z$, and $r:X\times Y\times Z\to X\times Z$ are the
projections. Furthermore, given four varieties we an obvious
isomorphism,
$$
\gamma_{X,Z,W}\cdot(\gamma_{X,Y,Z}\times I_{\mathbf{C}(Z,W)})\simeq
\gamma_{X,Y,W}\cdot(I_{\mathbf{C}(X,Y)}\times\gamma_{Y,Z,W})
$$
\end{lemma}
\begin{proof}
It is a straightforward exercise using the definition of $\mathbf{C}
(X,Y)$ to check both that $r_*(p^*({\mathcal
G}^\cdot)\otimes_{\mathcal{O} _Y}q^*({\mathcal F}^\cdot))$ is indeed
an object in ${\mathbf C} (X,Z)$ and that $ \gamma_{X,Y,Z}$ is
biexact.
\end{proof}

\subsection{Enriching the category of varieties over the category of chain complexes}
It will be convenient to work not with $K$-theory spectra but with
\emph{chain complexes} which compute rational $K$-theory. Recall
that if $\mathbf{A}$ is a spectrum, then
$\pi_*(\mathbf{A})\otimes\mathbb{Q}\simeq H_*(\mathbf{A},\mathbb{Q}
)$. If $\mathcal E$ is an exact category then, following
\cite{McCarthy}, we have an explicit chain complex
$M_*(\mathcal{E})$ which computes
$H_*(\mathbf{K}(\mathcal{E}),\mathbb{Q})$.  It is the complex
associated  to the cubical object which in degree $n$ is the
$\mathbb{Q}$-vector space spanned by the set of $n$-cubes of exact
sequences. (See \cite{Takeda} for a nice exposition of this
construction.) This construction is functorial in the sense that an
exact functor $\gamma:\mathcal{E}\to \mathcal{F}$  induces a map
of chain complexes $M_*(\gamma):M_*(\mathcal{E})\to
M_*(\mathcal{F})$, and that an isomorphism between functors
$\gamma\to\gamma'$ induces a homotopy between $M_*(\gamma)$ and
$M_*(\gamma')$.

One important feature of this construction is that if $\mathcal E$,
$\mathcal F$ and $\mathcal{G}$ are exact categories, and
$\mu:\mathcal{E}\times\mathcal{F}\to\mathcal{G}$ is a biexact
functor, then given an $m$-cube $\alpha$ of exact sequences in
$\mathcal{E}$ and an $n$-cube $\beta$ of exact sequences
in $\mathcal{F}$,
$\mu(\alpha,\beta)$ is an $(m+n)$-cube of exact sequences in
$\mathcal G$. Clearly if either
$\alpha$ or $\beta$ is degenerate, their product is too, and thus we have a pairing:
$$M_*(\mu):M_*(\mathcal{E})\otimes_\mathbf{Q}M_*(\mathcal{F})\to M_*(\mathcal{G})\;.$$

Suppose that $\mathcal{E}$ is  an exact category and $\mathbf{w}$
is a subcategory of $\mathcal{E}$ so that $(\mathcal{E},\mathbf{w})$
is a category with cofibrations (= to the admissible monomorphisms in
$\mathcal{E}$) and weak equivalences. Then we have a fibration
sequence of $K$-theory spectra
$$ \mathbf{K}(\mathcal{E}^\mathbf{w})\to \mathbf{K}(\mathcal{E})\to \mathbf{K}(\mathbf{w}\mathcal{E})  $$
\cite{Waldhausen}, 1.6.4 and \cite{Thomason_Trobaugh}, 1.8.2. Observe
that $M_*(\mathcal{E}^\mathbf{w})$ is a subcomplex of
$M_*(\mathcal{E})$, since a degenerate cube of exact sequences in
$\mathcal{E}$ which lies in $\mathcal{E}^\mathbf{w}$ is degenerate
as a cube in $\mathcal{E}^\mathbf{w}$. Hence we have:
\begin{lemma}
The rational K-theory of the Waldhausen category
$(\mathcal{E},\mathbf{w})$ is computed by the chain complex
$$M_*(\mathbf{w}\mathcal{E})= M_*(\mathcal{E})/M_*(\mathcal{E}^\mathbf{w})\; .$$
Furthermore, $(\mathcal{E},\mathbf{w})\mapsto
M_*(\mathbf{w}\mathcal{E})$ is a covariant functor from exact
categories to chain complexes.
\end{lemma}

\begin{defn}
If $X$ and $Y$ are varieties, we shall write $M_*(X,Y)$ for  $M_*({\mathbf
C} (X,Y))$. We then have a functor:
\begin{eqnarray*}
{\mathcal V}^{\rm op} \times {\mathcal V} &\to & {\rm Chain(\mathbb{Q})
} \\
(X,Y) &\mapsto & M_*(X,Y)
\end{eqnarray*}
where ${\rm Chain}(\mathbb{Q})$ is the category of chain complexes
of $\mathbb{Q}$-vector spaces.
\end{defn}

\begin{prop}
\label{G} If $X$ is regular and projective, the natural functor from
${\mathbf C} (X,Y)$ to the Waldhausen category of all bounded complexes
of quasi-coherent sheaves with coherent cohomology on $X \times Y$
induces an
isomorphism on $K$-theory, and hence a weak equivalence of
$K$-theory spectra:
$$
{\mathbf K}C (X,Y) \simeq {\mathbf G}(X \times Y)\; .
$$
\end{prop}

\begin{proof} By \cite{Thomason_Trobaugh} 1.9.7 and 1.9.8
it is enough to show:
\begin{itemize}
\item[i)] Any bounded complex ${\mathcal G}^{\cdot}$ of quasi-coherent sheaves on
$X \times Y$ is isomorphic in the derived category  to a
bounded complex ${\mathfrak F}^{\cdot}$ of universally strongly acyclic quasi-coherent sheaves with each ${\mathfrak F}^i$
flat over $X$.
\item[ii)] Any bounded complex ${\mathcal G}^{\cdot}$ of quasi-coherent sheaves on
$X \times Y$ with coherent cohomology is quasi-isomorphic to a
bounded complex of coherent sheaves.
\end{itemize}

Assertion ii) is a standard fact, since any quasi-coherent sheaf is
a direct limit of coherent sheaves.

To prove i), we have to deal with two issues, flatness and universal strong acyclicity.
The approach we take addresses both issues simultaneously.
If $V$ is a quasi-compact separated scheme, and
${\mathfrak U} = \{ U_i \}_{i=1}^n$ is a finite affine cover of $V$,  $\mathfrak{U}$ determines a diagram in the category of affine schemes,
with vertices the non-empty intersections $U_I$ ($\emptyset\neq I\subset \{1,\ldots,n\}$) of opens in the cover and arrows the inclusions.
Following \cite{Alonso_Lopez_Lipman} we define a quasi-coherent $\mathfrak{U}$-module $\mathfrak{M}$ to be a
family $M_I$ of $\mathcal{O}_V(U_I)$-modules, together
with maps $M_I\to M_J$ if $I\subset J$ which are $\mathcal{O}_V(U_J)$ linear and which satisfy the obvious presheaf condition.
A sequence of $\mathfrak{U}$-modules is exact if the induced sequences of $\mathcal{O}_V(U_I)$ modules are exact for all $I$, and a
$\mathfrak{U}$-module $\mathfrak{M}$ is flat if all the $M_I$ are flat.

There is an obvious exact functor from quasi-coherent sheaves on $V$ to quasi-coherent ${\mathfrak U}$-modules,
sending a quasi-coherent sheaf $\mathcal F$ to the ${\mathfrak U}$-module  $U_I\mapsto \mathcal{F}(U_I)$. On the other hand,
given a quasi-coherent $\mathfrak{U}$-module $\mathfrak{M}$, there
is a $\check{\rm C}$ech complex $\check{\mathcal{C}}^*(\mathfrak{M})$ which is a
bounded complex of quasi-coherent sheaves, and which is functorial with respect to $\mathfrak{M}$.
Since the sheaves in the complex are all direct sums of direct images of quasi-coherent sheaves on affine opens they are universally strongly acyclic.
If $\mathfrak{F}$ is the $\mathfrak{U}$-module associated to a quasi-coherent sheaf $\mathcal{F}$ on $V$, then
$\check{\mathcal{C}}^*(\mathfrak{F})$ is just the standard $\check{\rm C}$ech resolution $\check{\mathcal{C}}^* ({\mathfrak U} , {\mathcal F})$
which is quasi-isomorphic to ${\mathfrak F}$.
It is shown in section (1.2) of \emph{op. cit} that the functor $\mathfrak{M}\mapsto \check{\mathcal{C}}^*(\mathfrak{M})$ is exact, and that
if $\mathfrak{M}$ is flat as a $\mathfrak{U}$-module,
then $\check{\mathcal{C}}^*(\mathfrak{M})$ is a complex of flat quasi-coherent sheaves on $V$.  Furthermore
any quasi-coherent $\mathfrak{U}$-module is a quotient of a \emph{flat} quasi-coherent $\mathfrak{U}$-module.
Suppose now that $\mathcal{G}^*$ is a bounded complex of quasi-coherent sheaves on $V$.
Let $\mathfrak{G}^*$ be the associated complex of quasi-coherent $\mathfrak{U}$-modules, and choose a resolution
$\mathfrak{F}^*$ of $\mathfrak{G}^*$ by \emph{flat} quasi-coherent $\mathfrak{U}$-modules.
If we then then apply the functor $\check{\mathcal{C}}^*$ we obtain a complex $\check{\mathcal{C}}^*(\mathfrak{F}^*)$ of
universally strongly acyclic flat quasi-coherent $\mathcal{O}_V$-modules quasi-isomorphic to $\mathcal{G}^*$.

Returning now to the situation of part ii) of the theorem, suppose that $\mathcal{G}^*$ is a
bounded complex of quasi-coherent sheaves of $X\times Y$-modules.
Since we assume that $X$ is regular, $\mathcal{O}_X$ has finite global tor-dimension.
Hence if $\mathfrak{F}^*$ is a flat resolution
of $\mathfrak{G}^*$ as above, for sufficiently large $k$, the kernel of $\mathfrak{F}^{-k}\to \mathfrak{F}^{-k+1} $ while not flat as an
$\mathcal{O}_{X\times Y}$ module, will be flat as an $\mathcal{O}_X$-module.  Hence  taking the ``good'' truncation of the
complex $\mathfrak{F}^*$ at that point  we obtain a finite  resolution $\overline{\mathfrak{F}}^*$
of $\mathfrak{G}^*$ by $\mathcal{O}_X$-flat $\mathfrak{U}$-modules, and
so on applying the functor $\check{\mathcal{C}}^*$ we obtain a bounded complex, isomorphic in the derived category to $\mathcal{G}^*$,
of universally strongly acyclic
quasi-coherent sheaves which are flat over $\mathcal{O}_X$.
\end{proof}

\begin{cor}
Suppose that $X$ is regular and projective, $Y$ is an arbitrary
variety and $Y' \subset Y$ is a closed subscheme with open
complement $U$.  There is map $M_*(X,Y) \to M_*(X,U)$ (since the
restriction of an object ${\mathfrak F}^{\cdot}$ in ${\mathbf C} (X,Y)$
to $X \times U$ is clearly in ${\mathbf C} (X,U)$). Then the associated
sequence of maps of complexes:
$$
M_*(X,Y') \to M_*(X,Y) \to M_*(X,U) \, .
$$
induces a quasi-isomorphism from the cone of the map 
$$M_*(X,Y') \to
M_*(X,Y)$$ 
to $M_*(X,U) $.

\end{cor}

\begin{proof} This follows from Quillen's localization
theorem, which implies that
$$
{\mathbf G} (X \times Y') \to {\mathbf G} (X \times Y) \to {\mathbf G} (X \times
U)
$$
is a fibration sequence of spectra. \end{proof}

\begin{defn}
Let $X$ be a variety, and $Y.$ a simplicial variety with proper face
maps. Then we define
$$
M_*(X,Y.) := {\rm Tot}_*(j\mapsto \, M_*(X , Y_j)) \, .
$$
\end{defn}

If $\beta. : Y. \to Z.$ is a morphism between simplicial varieties
with proper maps which is proper in each degree, we define
$$
M_*(X , \beta.) := \mbox{\rm cone} \, (M_*(X,Y.) \to M_*(X,Z.)) \, .
$$

If $X.$ is a simplicial variety, we define (with notation as above)
$$
M_*(X. , Y.) := {\rm Tot}^{\prod}_*\,(i\mapsto M_*(X_i , Y.))\; ,
$$
\emph{i.e.}, $M_k(X. , Y.)=\prod_i M_{i+k}(X_i , Y.))$, and
$$
M_*(X. , \beta.) := \mbox{\rm cone} \, (M_*(X.,Y.) \to M_*(X.,Z.))
\, .
$$
Note that since infinite products are exact, $M_*(X. , \beta.)$
 is isomorphic to
$$ {\rm Tot}^{\prod}_*\,(i\mapsto M_*(X_i , \beta.))\; .$$
If $\alpha : W. \to X.$ is an arrow between simplicial
varieties, we set
$$
M_*(\alpha. , \beta.) := \mathrm{cone} \, M_*(X. , \beta.) \to
M_*(W. , \beta.))[-1] \, .
$$

These constructions are all contravariant in the first variable, and
covariant in the second variable ({\it e.g.} with respect to
morphisms between arrows of simplicial objects in the category of
proper maps between varieties). There are also natural equivalences
when two of these constructions should agree ; for example if $X. =
X$ is a constant simplicial variety, the natural map
$$
M_*(X,Y.) \to M_*(X. , Y.)
$$
is a quasi-isomorphism.

Suppose that $X.$, $Y.$, $Z.$ are all simplicial proper varieties.
Then we have a composition:
$$M_*(X.,Y.)\otimes M_*(X.,Y.)\to M_*(X.,Z.) $$
which is associative up to homotopy, defined by
$$
M_i(X.,Y.)\otimes M_m(Y.,Z.)
\to M_{i+m}(X.,Z.)\\
$$
\begin{multline*}
\prod_j\left(\bigoplus_{k+l=j+i}M_l(X_j,Y_k)\right)\otimes
\prod_n\left(\bigoplus_{p+q=m+n}M_p(Y_n,Z_q)\right)\\
\to \prod_r\left(\bigoplus_{t+s=i+m+r}M_t(X_r,Z_s)\right)
\end{multline*}
which for $l+p=t$, $k$ and $s=q$, is induced by the pairings
$$
M_l(X_j,Y_k)\otimes M_p(Y_n,Z_q)\to M_{t}(X_r,Z_s)
$$
and is zero otherwise.

It is straightforward to check that this defines a category enriched
over the category of chain complexes of vector spaces (with respect
to the monoidal structure given by tensor product), with objects
simplicial varieties.

\begin{lemma}\label{M-descent}
Let $\pi. : Y'. \to Y.$ be a proper hypercover of simplicial
varieties with proper faces maps. Then, if $X.$ is a simplicial
regular projective variety, the map
$$
M_*(X. , Y'_{\cdot}) \to M_*(X.,Y.)
$$
is a quasi-ismorphism. Similarly, if $\pi : \beta'_{\cdot} \to
\beta.$ is a map of arrows
$$
\xymatrix{
Y'_{\cdot} \ar[d]_{\pi_Y} \ar[r]^{\beta'_{\cdot}} &Z'_{\cdot} \ar[d]^{\pi_Z} \\
Y. \ar[r]^{\beta.} &Z. \\
}
$$
with $\pi_Y$ and $\pi_Z$ proper hypercovers, we have a weak
equivalence, for all arrows $\alpha.$ between simplicial regular
projective varieties
$$
M_*(\alpha. , \beta'_{\cdot}) \to M_*(\alpha. , \beta) \, .
$$
\end{lemma}
\begin{proof}
For a fixed $X_i$ we know that
$M_*(X_i , Y'_{\cdot}) \to M_*(X_i,Y.)$ is a quasi-isomorphism, by theorem \ref{descent_for_Gthy}.
Since $M_*(X. , Y.)$ is the inverse limit, over $n\geq 0$, of ${\rm
Tot}^{\prod}_*\,(i\mapsto M_*(X_i , Y.),i\leq n)$ a standard
$\lim^1$ argument shows that $M_*(X. , Y'_{\cdot}) \to M_*(X.,Y.)$
is a quasi-isomorphism.  A similar argument gives the second
assertion of the lemma.
\end{proof}

\begin{lemma}
Suppose that $Z.$ is a simplicial variety with proper face maps, and
$Y. \subset Z.$ is a closed simplicial subvariety such that the
complement $U. = Z. \backslash Y.$ is an open simplicial subvariety.
Then, for all simplicial regular projective varieties $X.$, the
localization sequences
$$
M_*(X_i , Y_j) \to M_*(X_i , Z_j) \to M_*(X_i , U_j)
$$
induce a fibration sequence
$$
M_*(X. , Y.) \to M_*(X. , Z.) \to M_*(X. , U.)
$$
and similarly, if $\alpha. : W. \to X.$ is an arrow in the category
of simplicial regular projective varieties, we have a (co)fibration
sequence:
$$
M_*(\alpha. , Y.) \to M_*(\alpha. , Z.) \to M_*(\alpha. , U.) \, .
$$
\end{lemma}

\begin{proof}
Again a standard $\lim^1$ argument.
\end{proof}

\begin{lemma}\label{6011}
Let $g : X \to Y$ be a proper morphism in ${\mathcal V}$, and
$\Gamma_g:X\to X \times Y$ its graph. Let ${\mathcal O}_{\Gamma_g}$
be the object $(\Gamma_g)_*(\mathcal{O}_X)$  in ${\mathbf C} (X,Y)$.
Then if $p : Y' \to Y$ is proper, $p_* ({\mathcal O}_{\Gamma_g}) =
{\mathcal O}_{\Gamma_{p \circ g}}$. (Note that this is an
\emph{equality} rather than an isomorphism!). While if $f:X' \to X$
is a morphism, we have a canonical isomorphism:
$$
f^* ({\mathcal O}_{\Gamma_g}) \simeq {\mathcal O}_{\Gamma_{f \circ
g}}\, .
$$
Hence we have a well defined object in the big Zariski site
over $X$, $(f:X' \to X)\mapsto  (\Gamma_{(g\cdot
f)})_*(\mathcal{O}_{X'})$ .
\end{lemma}
\begin{proof} Exercise.
\end{proof}

On the category ${\rm Chain}(\mathbb{Q})$ of chain complexes
concentrated in positive degrees, $\mathrm{H}_0$ gives a functor to
the category of $\mathbb{Q}$-vector spaces, viewed as chain
complexes concentrated in degree zero, together with a natural
transformation
$$
\eta : {\rm Id} \to \mathrm{H}_0 \; .
$$
This is the homological equivalent of taking the 0-th stage of the
Postnikov tower of a spectrum.

Hence, given a pair of objects $X.$, $Y.$ in $RP^{\Delta^{\rm op}}$,
we get a map of cosimplicial simplicial chain complexes:
$$
(i,j) \mapsto (M_*(X_i , Y_j) \to {KC}(X_i , Y_j))
$$
since $\mathrm{H}_0(M_*(X_i , Y_j))\simeq  {KC}(X_i , Y_j)$, and
therefore we have a homomorphism
$$
\mathrm{H}_0(M_*(X. , Y.)) \to \mathrm{H}_0(\mathrm{Tot}_*(KC(X_i ,
Y_j))\; .
$$
But the target of this homomorphism is  equal to the homotopy
classes of maps $\Gamma_* (X.) \to \Gamma_* (Y.)$ (cf. \S 5.3). More generally,
given arrows $\alpha.$ and $\beta.$ in $RP^{\Delta^{\rm op}}$, we
get in a similar fashion a morphism
$$
\gamma : \mathrm{H}_0(\mathrm{M}_* (\alpha. , \beta.)) \to \{
\mbox{Homotopy classes of maps} \ \Gamma_* (\alpha.) \to \Gamma_*
(\beta.)\} \, .
$$

\begin{prop}
\label{truc2} Let $f : \alpha. = (W. \overset{a}{\longrightarrow}
X.) \to \beta. = (Y. \overset{b}{\longrightarrow} Z.)$ be a map of
arrows in $RP^{\Delta^{\rm op}}$, so that we have a commutative
diagram :
$$
\xymatrix{
W. \ar[d]_f \ar[r]^a &X. \ar[d]^g \\
Y. \ar[r]^b &Z. }
$$
of maps of simplicial regular projective varieties. Then there is a
$0$-cycle $[f]$ in $M_*(\alpha. , \beta.)$ such that $\gamma \, [f]$
is the homotopy class of the map $\Gamma_* (f)$.
\end{prop}

\begin{proof} It follows from Lemma~\ref{6011}, that the
sheaves ${\mathcal O}_{\Gamma_{f_i}}$ and ${\mathcal
O}_{\Gamma_{g_j}}$ determine elements $[{\mathcal
O}_{\Gamma_{f_i}}]$ and $[{\mathcal O}_{\Gamma_{g_j}}]$ in $M_0(W_i
, Y_i)$ and $M_0(X_j , Z_j)$, respectively, such that for all $i$,
$b_* [{\mathcal O}_{\Gamma_{f_i}}] = a^* [{\mathcal
O}_{\Gamma_{g_i}}]$, and which are compatible with all face and
degeneracy maps, and therefore define a $0$-cycle $[f]$ in
$M_*(\alpha. , \beta.)$. On the other hand, under the map $\gamma$,
for each $i$, the class of $[{\mathcal O}_{\Gamma_{f_i}}]$ in $K_0
\, C (W_i , Y_i)$ is the map $\Gamma_* (f_i)$, and similarly for
$\Gamma_{g_j}$, and so $\gamma \, [f] = \Gamma_* (f)$.
\end{proof}

Let ${\mathbf C}' (X,Y)$ be the Waldhausen category of bounded complexes
of quasi-coherent sheaves of ${\mathcal O}_{X \times Y}$ modules
which are flat over $X$, and which have coherent cohomology having
support proper over $Y$.

\begin{lemma}
\begin{itemize}
\item[{\rm 1)}] The inclusion of Waldhausen categories 
$${\mathbf C}(X,Y)\subset{\mathbf C}'(X,Y)$$ 
induces an isomorphism on $K$-theory.
\item[{\rm 2)}] ${\mathbf C}' (X,Y)$ is contravariant with respect to $X$ : given an object ${\mathfrak F}^{\cdot}$ in ${\mathbf C}' (X,Y)$
 and a morphism $f : X' \to X$, $f^* ({\mathfrak F}^{\cdot}) \in {\mathbf C}' (X' , Y)$.
\end{itemize}
\end{lemma}

\begin{proof} Straightforward.
\end{proof}

We write $M'_* (X,Y)$ for the complex $M_*({\mathbf C}'(X,Y))$ computing
the rational $K$-theory of ${\mathbf C}'(X,Y)$ . Suppose that $f : X \to
Y$ is a flat morphism of varieties. Pick a compactification $i : Y
\hookrightarrow \overline{Y}$ with complement $Z$. Note that $i
\cdot f : X \to \overline{Y}$ is also flat. The structure sheaf of
the graph of $f$, ${\mathcal O}_{\Gamma_f}$, is an object in ${\mathbf
C}' (X,Y)$ while ${\mathcal O}_{\Gamma_{i \circ f}}$ is an object in
${\mathbf C}' (\overline{Y} , X)$. Notice that the pull back of
${\mathcal O}_{\Gamma_{i \circ f}}$ to ${\mathbf C}' (Z,X)$ is
identically zero. Hence $[{\mathcal O}_{\Gamma_{i \circ f}}]$ gives
a $0$-cycle in the homotopy fibre of $M'_* (\overline{Y} , X) \to
M'_* (Z,X)$, which we denote $M'_* ((Z \to \overline{Y}) , X)$. Now
pick a non-singular hypercover $\pi : (\widetilde Z. \to \widetilde
Y.) \to (Z \to \overline{Y})$. Again by contravariance of ${\mathbf C}'
(\cdot , \cdot)$ in the first variable, we get a class (which we
denote $\pi^* ([{\mathcal O}_{\Gamma_{i \circ f}}]))$ in
$\mathrm{H}_0 \, M'_* ((\widetilde Z. \to \widetilde Y.) , X) \simeq
\mathrm{H}_0 \, M_* ((\widetilde Z. \to \widetilde Y.) , X)$

Now pick a compactification $\overline{X}$ of $X$ with complement
$W$. By localization, we have a fibration sequence, writing $\alpha.
= (\widetilde Z. \to \widetilde Y.)$,
$$
M_* (\alpha , W) \to M_* (\alpha , \overline{X}) \to M_* (\alpha ,
X)
$$
and hence a class in
$$
\zeta \in \mathrm{H}_0 (M_* (\alpha , (W \to \overline{X}))
$$
which maps to $\pi^* ([{\mathcal O}_{\Gamma_{i \circ f}}])$.

If $p : \beta = (\widetilde W. \to \widetilde X.) \to (W \to
\overline{X})$ is a regular proper hypercover of the arrow $(W \to
\overline{X})$, by descent (\ref{descent}), we have that
$$
p_* : M_* (\alpha , \beta) \to M_* (\alpha , (W \to \overline{X}))
$$
is a quasi-isomorphism, and hence we get a class $p_* (\zeta)^{-1} \in
\mathrm{H}_0 (M_* (\alpha , \beta))$. Now applying $\gamma$, we get
a homotopy class of maps $\Gamma_* (\alpha) \to \Gamma_* (\beta)$,
{\it i.e.} a map $f^* : h(Y) \to h(X)$.

We can compare this construction with the one situation in which we
have already defined a pull back map, \emph{i.e.} when $f:U\to X$ is
an open immersion:

\begin{prop}
Let $i:U\to X$ is an open immersion between varieties. Then the
associated map of weight complexes $h(X)\to h(U)$ which comes from
the definition of weight complexes, agrees with the map induced by
viewing ${\mathcal O}_{\Gamma_i}$ as being quasi-isomorphic to an
object of ${\mathbf C} (U,X)$.
\end{prop}
\begin{proof} Recall that the map of weight complexes is
constructed as follows. Pick a compactification $\overline{X}$ of
$X$, and set $Y=\overline{X}\backslash X$ and
$Z=\overline{X}\backslash U$, so that $Y\subset Z\subset X$, with
all three proper varieties. We can find a diagram:
$$
\begin{CD}
\widetilde{Y}.@>>> \widetilde{Z}.@>>> \widetilde{X}. \\
@VVV      @VVV      @VVV \\
{Y}@>>> {Z}@>>> {X}
\end{CD}
$$
with $\widetilde{Y}\to Y$, $ \widetilde{Z}\to Z$ and
$\widetilde{X}.\to X$  all non-singular proper hypercovers. Consider
the induced commutative square:
$$
\begin{CD}
\widetilde{Y}.@>{\alpha}>> \widetilde{X}.\\
@V{j.}VV                 @V{1_{\widetilde{X}.}}VV   \\
\widetilde{Z}.@>{\beta}>> \widetilde{X}.
\end{CD}
$$
Following Proposition \ref{truc2}, we know that there is a
corresponding 0-cycle $M_*(\alpha, \beta)$ which induces the
corresponding map of weight complexes. By localization and descent,
$M_*(\alpha, \beta)\simeq M_*(\alpha., U)$ and by the compatibility
of the objects ${\mathcal O}_{\Gamma_g}$ of lemma \ref{6011} with
push-forward and pull-back on the source and target of $f$, we see
that the image of this zero cycle is exactly the 0-cycle used to
define pull back with respect to the flat morphism
 $i:U\to X$.
\end{proof}

\begin{thm}
Let $X,Y$ be projective varieties. Then there is a map
$$
\gamma : K_0({\mathbf C}(X,Y))_\mathbb{Q}\simeq \mathrm{H}_0(M_*(X,Y))
\to \mathrm{Hom}(h(X) , h(Y))
$$
which is covariant with respect to $Y$ and contravariant with
respect to $X$, and which is the identity if $X$ and $Y$ are regular
and projective (note that then $K_0 ({\mathbf C} (X,Y)) = G_0 (X \times
Y)$).
\end{thm}

\begin{proof} Let $p : \widetilde X. \to X$ and $q :
\widetilde Y. \to Y$ be non-singular proper hypercovers. Given
$\varphi \in K_0({\mathbf C}(X,Y))$, we have $p^* (\varphi) \in
\mathrm{H}_0 (M_* (\widetilde X. , Y))$ and $q_* : M_* (\widetilde
X.,\widetilde Y.) \to M_*(\widetilde X. , Y)$ is a weak equivalence.
Hence there is a unique $\widetilde\varphi \in
\mathrm{H}_0(M_*(\widetilde X. , \widetilde Y.))$ such that $q_*
(\widetilde\varphi) = p^* (\varphi)$.

Applying the natural transformation $\gamma$ defined before
\ref{truc2}, we get an element
$$
\gamma (\widetilde\varphi) \in H_0 ({\rm Hom} (\Gamma_* (\widetilde
X.) , \Gamma_* (\widetilde Y))) \, .
$$
That the homotopy class of this map does not depend on the choice of
$p : \widetilde X. \to X$ and $q : \widetilde Y. \to Y$ follows the
same pattern as \ref{simplicial}. Functoriality with respect to $X$
and $Y$ is an immediate consequence of the fact that the complex
$M_* (X,Y)$ is contravariant with respect to $X$ and covariant with
respect to $Y$. If $X$ and $Y$ are already regular, we take $p$ and
$q$ to be the identity and the assertion is clear. \end{proof}

Notice that it follows from this result that if $f : Y \to X$ is a
morphism of finite tor-dimension between projective varieties, then
there is a well defined map
$$
f^* : h(X) \to h(Y) \, ,
$$
since ${\mathcal O}_{\Gamma_f}$ is quasi-isomorphic to an object of
${\mathbf C} (X,Y)$.



\begin{thebibliography}{999}
\bibitem{Alonso_Lopez_Lipman}
{\sc Alonso Tarr\'io}, L., {\sc Jerem\'as L\'opez}, A. and {\sc Lipman}, J.,
Local homology and cohomology on schemes,
{\it Ann. Sci. \'Ecole Norm. Sup. (4)}  {\bf 30} (1997), 1--39.
\bibitem{Bittner} {\sc Bittner}, F., The universal Euler characteristic for varieties of characteristic zero.  {\it Compos. Math.}  {\bf 140}  (2004),  no. 4, 1011--1032
\bibitem{Borel-Serre} {\sc Borel}, A., {\sc Serre}, J.-P.,
Le th\'eor\`eme de Riemann-Roch. 
{\it Bull. Soc. Math. France} {\bf 86} (1958) 97--136.
\bibitem{BK} {\sc Bousfield} A.K. and {\sc Kan} D.M.
Homotopy limits, completions and localizations. {\it Lecture Notes
in Mathematics} {\bf 304}, Berlin-Heidelberg-New York:
Springer-Verlag.
\bibitem{Conrad} {\sc Conrad} B., Keele-Mori theorem via stacks, unpublished notes, http://www.lsa-unich.edu/$\sim$bdconrad/coarsespace.pdf
\bibitem{de-jong-fourier} {\sc de Jong}, A. J., Families of curves and alterations, {\it Ann. Inst. Fourier } {\bf 47}, 2 (1997), 599--621.
\bibitem{Deligne-Hodge-III} {\sc Deligne}, P., Th\'eorie de Hodge III, {\it Inst. Hautes \'Etudes Sci. Publ. Math.} {\bf 44} (1974), 5--77.
\bibitem{Gabber} Gabber,O., Finiteness theorems for �etale cohomology of excellent schemes,
Conference in honor of P. Deligne on the occasion of his 61st birthday, IAS, Princeton, October
2005.
\bibitem{Gabriel-Zisman} {\sc Gabriel}, P. and {\sc Zisman},  M., Calculus of fractions and homotopy theory, {\it Ergebnisse der Mathematik und ihrer Grenzgebiete}, {\bf 35}, Springer-Verlag New York, Inc., New York (1972).
\bibitem{Geisser-Hesselholt} {\sc Geisser}, T., and {\sc Hesselholt}, L.,
Topological cyclic homology of schemes, {\it in} Algebraic
$K$-theory (Seattle, WA, 1997), \emph{Proc. Sympos. Pure Math.},
\textbf{67}, Amer. Math. Soc., Providence, RI, (1999), 41--87.
\bibitem{Gillet-RR} {\sc Gillet}, H., Riemann-Roch Theorems for Higher Algebraic $K$-Theory, {\it Adv. in Math.} {\bf 40} No. 3 (1981), 203-289.
\bibitem{Gillet-homological-descent} {\sc Gillet}, H., Homological descent for the {$K$}-theory of coherent sheaves, {\it in} Algebraic $K$-theory, number theory, geometry and analysis (Bielefeld, 1982), {\it Lecture Notes in Math.} {\bf 1046} (1984) 80-103, Springer, Berlin.
\bibitem{Gillet-Intersection-theory-on-stacks} {\sc Gillet}, H., Intersection theory on algebraic stacks and {$Q$}-varieties, {\it in}
Proceedings of the Luminy conference on algebraic $K$-theory
(Luminy, 1983), {\it J. Pure Appl. Algebra} {\bf 34} (1984),
193-240.
\bibitem{Gillet-Soule-motives-descent} {\sc Gillet}, H. and {\sc Soul\'e}, C.,  Descent, motives and $K$-theory, {\it J. Reine Angew. Math.} {\bf 478} (1996), 127-176.
\bibitem{EGAIII} {\sc Grothendieck}, A., \'El\'ements de g\'eom\'etrie alg\'ebrique. IV. \'Etude locale des sch\'emas et des morphismes de sch\'emas. III. {\it Inst. Hautes \'Etudes Sci. Publ. Math.} {\bf 28} (1966), p. 255.
\bibitem{SGA4-2} {\sc Grothendieck}, A. and others, Th\'eorie des topos et cohomologie \'etale des sch\'emas. Tome 2. S\'eminaire de G\'eom\'etrie Alg\'ebrique du Bois-Marie 1963-1964 (SGA 4), Dirig\'e par M. Artin, A. Grothendieck et J. L. Verdier. Avec la collaboration de N. Bourbaki, P. Deligne et B. Saint-Donat, {\it Lecture Notes in Mathematics} {\bf 270}, Springer-Verlag, Berlin, (1972).
\bibitem{Guillen} {\sc Guill\'en}, F., {\sc Navarro Aznar}, V.,
Un crit\`ere d'extension d'un foncteur d\'efini sur les sch\'emas lisses
 {\it Inst. Hautes \'Etudes Sci. Publ. Math.} {\bf 95} (2002), 1-91. 
\bibitem{Hartshorne} {\sc Hartshorne}, R., Algebraic geometry, {\it Graduate Texts in Mathematics}, No. 52, Springer-Verlag, New York (1977).
\bibitem{Hovey-Shipley-Smith} {\sc Hovey}, M., {\sc Shipley}, B. and {\sc Smith}, J., Symmetric spectra, {\it J. Amer. Math. Soc.} {\bf 13} No 1 (2000), 149-208.
\bibitem{JardinePresheavessymmetricspectra} {\sc Jardine}, J. F. Presheaves of symmetric spectra, {\it J. Pure Appl. Algebra} {\bf 150} No. 2 (2000), 137-154.
\bibitem{KS} {\sc Kahn}, B. and {\sc Sujatha}, R.,  A few localisation theorems, {\it  Homology, Homotopy and Applications} {\bf 9} (2007), 137-161.
\bibitem{Keele-Mori} {\sc Keele}, S. and {\sc Mori}, S., Quotients by groupoids, {\it Ann. Math.} {\bf 145} (1997), 193-213.
\bibitem{Laumon-M-B-stacks} {\sc Laumon}, G. and {\sc Moret-Bailly}, L., Champs alg\'ebriques, {\it Ergebnisse der Mathematik und ihrer Grenzgebiete. $3$. }
 {\bf 39}, Springer-Verlag, Berlin (2000).
\bibitem{MacLane-Categories-working-math} {\sc Mac Lane}, S., Categories for the working mathematician, {\it Graduate Texts in Mathematics} {\bf 5}, Second Edition, Springer-Verlag, New York (1998).
\bibitem{Manin}{\sc Manin}, Y. Correspondences, motifs and monoidal transformations
{\it Mat. Sb., N.} {\bf 77}(119) (1968),  475-507.
\bibitem{McCarthy} {\sc McCarthy}, R., A chain complex for the spectrum homology of the algebraic $K$-theory of an exact
category. \emph{in} Algebraic $K$-theory (Toronto, ON, 1996),
199--220, \emph{Fields Inst. Commun.}, 16, Amer. Math. Soc.,
Providence, RI, 1997.
\bibitem{Quillen-Homotopical-Algebra} {\sc Quillen}, D. G., Homotopical algebra, {\it Lecture Notes in Mathematics},{ \bf 43}, Springer-Verlag, Berlin (1967).
\bibitem{QuillenHAKTI} {\sc Quillen}  D. G., Higher Algebraic K-theory I, {\it Lecture Notes in Mathematics},
 {\it Lecture Notes in Mathematics}{\bf 341} (1973), 85-147.
\bibitem{Rost-Chow-groups-with-coefficients} {\sc Rost}, M., Chow groups with coefficients, {\it Doc. Math.} {\bf 1} (1996), No. 16, 319-393.
\bibitem{Takeda} {\sc Takeda}, Y., Complexes of exact Hermitian cubes and the Zagier conjecture, \emph{Math. Ann.} 328 (2004), no. 1-2, 87--119.
\bibitem{AKTEC} {\sc Thomason}, R. W. Algebraic $K$-theory and \'{e}tale cohomology.
\emph{Ann. Sci. \'{E}cole Norm. Sup.} (4) \textbf{18} (1985), no. 3,
437--552.
\bibitem{Thomason_Trobaugh} {\sc Thomason}, R. W. and {\sc Trobaugh}, T., Higher algebraic $K$-theory of schemes and of derived categories, {\it in} The Grothendieck Festschrift, Vol. III, {\it Progr. Math.} {\bf 88} (1990), 247-435, Birkh\"auser Boston, Boston, MA.
\bibitem{Waldhausen} {\sc Waldhausen}, F., Algebraic $K$-theory of spaces, \emph{in}
Algebraic and geometric topology (New Brunswick, N.J., 1983),
318--419, \emph{Lecture Notes in Math.}, 1126, Springer, Berlin,
1985.
\end{thebibliography}
\end{document}